\documentclass{cmslatex}
\usepackage[paperwidth=7in, paperheight=10in, margin=.875in]{geometry}
 \usepackage[backref,colorlinks,linkcolor=red,anchorcolor=green,citecolor=blue]{hyperref}
\usepackage{amsfonts,amssymb}
\usepackage{amsmath}
\usepackage{graphicx}
\usepackage{cite}
\usepackage{enumerate}
\usepackage{algorithmicx}
\usepackage{algorithm}
\usepackage{verbatim}
\renewcommand{\theequation}{\arabic{section}.\arabic{equation}}

\def\L{\Lambda}

\def\cL{\mathcal{L}}
\DeclareMathOperator{\tr}{tr}

   \allowdisplaybreaks
\begin{document}
 \title{Numerical methods for stochastic differential equations based
 	on Gaussian mixture\thanks{Received date, and accepted date (The correct dates will be entered by the editor).}}


          \author{Lei Li\thanks{School of Mathematical Sciences, Institute of Natural Sciences, MOE-LSC, Shanghai Jiao Tong University, Minhang District, Shanghai 200240, China (leili2010@sjtu.edu.cn).} \and Jianfeng Lu\thanks{Department of Mathematics, Department of Chemistry and Department of Physics, Duke University, Durham, NC 27708, USA (jianfeng@math.duke.edu).}
          	\and Jonathan C. Mattingly\thanks{Department of Mathematics and Department of Statistical Science, Duke University, Durham, NC 27708, USA (jonm@math.duke.edu).}
          	\and Lihan Wang\thanks{Department of Mathematics, Duke University, Durham, NC 27708, USA (lihan@math.duke.edu).}}

         \pagestyle{myheadings} \markboth{GAUSSIAN MIXTURE METHODS FOR SDEs}{L. Li, J. Lu, J. Mattingly and L. Wang} \maketitle

  \begin{abstract}
  We develop in this work a numerical method for stochastic differential equations (SDEs) with weak second-order accuracy based  on Gaussian mixture.  Unlike conventional higher order schemes for SDEs based on It\^o-Taylor expansion and iterated It\^o  integrals, the scheme  we propose approximates the probability measure $\mu(X^{n+1}\mid X^n=x_n)$ using a mixture of Gaussians. The solution at the next time step $X^{n+1}$ is drawn from the Gaussian mixture with complexity linear in dimension $d$. This provides a new strategy to construct efficient high weak order numerical schemes for SDEs. 
\end{abstract}

          \section{Introduction}Stochastic differential equations (SDEs) \cite{oksendal03} have been
          used to model a wide range of phenomena, such as stock prices of
          financial derivatives \cite{bs73,lambertonlaperyre2007}, and physical
          systems in contact with heat bath
          \cite{van1992,coffey2012,liliulu2017}.  SDEs have recently also
          been used for analyzing stochastic gradient descent (SGD) in machine
          learning \cite{lte17,llql17,feng2017}.  The SDEs are dynamical systems
          with noise \cite{oksendal03,weinan2019applied} that often represent
          interactions that are not included in the model but affect the
          dynamics. For example, in the Langevin equations
          	\cite{coffey2012,pavliotis2014stochastic}, one considers the
          	evolution of a subsystem, while the rest of the system, consisting of
          	potentially large degrees of freedom, is regarded as a heat bath. The
          	interaction between the heat bath and the subsystem is modelled by
          	noise and dissipation terms.
          
         We will consider general SDEs driven by white noise
          	\cite{pavliotis2014stochastic,weinan2019applied} and in particular
          	SDEs in It\^o sense \cite[Chap. 5]{oksendal03}:
          	\begin{equation}\label{eq:sde}
          		dX(t)=b(X(t))dt+\sigma(X(t))dW,~~X(0)=x,
          	\end{equation}
          	where $X\in \mathbb{R}^d$, $W$ is a standard $m$-dimensional
          	Brownian motion, $b: \mathbb{R}^d\to \mathbb{R}^d$ is the drift, and $\sigma: \mathbb{R}^d\to \mathbb{R}^{d\times m}$ is the diffusion matrix. We are interested in its numerical approximations, and depending on whether our goal is to approximate the sample paths or the
          	distributions, the numerical methods can be classified into strong
          	schemes and weak schemes \cite{kp92,weinan2019applied}. The weak schemes approximate the distributions, and we refer the readers to Section \ref{subsec:weakconv} for more details. Indeed, weak schemes attempt to match the moments of the iterated It\^o integrals, and therefore, the key question for designing weak schemes is how to approximate these moments efficiently.  The classical Euler-Maruyama scheme \eqref{eq:em} is known to be a weak first-order scheme. In applications, schemes with higher order accuracy are often desired. The weak second-order schemes, however, are not trivial for SDEs. The traditional second-order schemes, based on It\^o-Taylor expansion \cite{kp92,milstein2013stochastic}, involve evaluations of the spatial derivatives of drift and diffusion coefficients, as well as iterated It\^o integrals. The weak second-order schemes date back to Milstein and Talay \cite{mil1979method,talay1984}. The Talay-Tubaro expansion can also yield a weak second-order approximation \cite{talay1990expansion}. In
          \cite{rossler2006runge,rossler2009second}, Runge-Kutta methods that achieve arbitrary weak order for scalar noise and weak second-order for general noise are developed. Runge-Kutta schemes can avoid
          approximating some derivatives of the drift and the diffusion
          coefficients directly 
          (see for example the Talay-Tubaro scheme), and lead to better
          stability.  A weak trapezoidal second-order method has been developed
          in \cite{am11}, which is derivative free and no evaluation of iterated
          It\^o integrals is needed. However, it leverages the structure of a
          particular, but common, class of equations. In \cite{acvz12}, higher
          order convergence for a class of SDEs is achieved based on solving
          modified SDEs. In \cite{nv08,nn09}, another class of higher order
          schemes were developed which are often referred to as Lie splitting
          methods. Like the current methods, they strive for a level of weak
          accuracy by deriving a condition which guarantees that certain terms
          vanish in an expansion of the difference of the true density and that
          of the numerical method. 
          
          
          As will be discussed in Section \ref{subsec:weakconv}, one may approximate the conditional distribution $\mu(X^{n+1}\mid X^n=x_n)$ with asymptotic weak local error $O(h^3)$ to achieve global
          	weak second-order accuracy, where $\{X^n\}$'s are the numerical solutions and $h$ is the step size. In this work, we propose a novel Gaussian mixture method to achieve $O(h^3)$
          	weak local error in which $X^{n+1}$ can be sampled from one of a mixture of Gaussians.
          Our ansatz is inspired by the expansion of the solution using commutators.
          Our Gaussian particle ansatz has two attractive properties: \begin{itemize}
          	\item Since only one of the Gaussians in the mixture is chosen in each step, the cost in each step is
          	minimal and the simulation is fast: we only need to generate $d$
          	scalar random variables (with three possible values only) to generate
          	an initial point, and then generate a $d$-dimensional multi-variate normal
          	variable whose mean and covariance matrix are related with the $d$ scalar random variables and obtained by
          	solving an ODE (see Remark \ref{rmk:numberofrv} for some comments on the number of random variables needed). In this sense, the method we construct in this paper can be considered as a random mixture of Euler-Maruyama steps that produces a higher order method, and it is simple to implement. 
          	
          	
          	\item Secondly, our scheme does not need the spatial derivatives of the
          	coefficients, which is useful in several contexts.
          \end{itemize}  
          Numerical simulations show that our Gaussian mixture method is indeed weak second-order for reasonable
          values of the step size. This agrees with our theoretical results that our methods are asymptotically second-order as the step size goes
          to zero. For related works about using Gaussian approximations for general distributions, see \cite{lsw16,acos2007}. In \cite{acos2007}, Gaussian processes based on a variational approach are used to approximate posterior measure in path space. In \cite{lsw17}, Gaussian approximation is used to approximate transition paths in Langevin dynamics.
          
          This work is primarily interested with accuracy considerations. This
          is reflected in the use of essentially forward-Euler like steps to
          builds the Gaussian Mixture. It is easy to imagine replacing these
          with an implicit step to address stability concerns. 
          
          The rest of the paper is organized as follows. In Section~\ref{sec:setup}, we give a brief introduction to SDEs and the basic setup of our problem. In particular, the concept and criteria for weak accuracy using test functions with bounded derivatives are introduced. 
          In Section~\ref{sec:gauss}, we introduce the idea of Gaussian mixtures for high order weak accuracy and develop an algorithm for one-dimensional SDEs with weak second-order accuracy, where the mean and variance of the Gaussian are computed either based on some ODEs or construction. In Section~\ref{sec:multidim}, we generalize the algorithm for 1D SDEs to SDEs in multi-dimensions. The number of Gaussian beams are exponential in the dimension $d$, but we only need $d$ discrete random variables to determine which beam we choose, so the complexity is linear in $d$. In Section~\ref{sec:numeric}, we perform several numerical examples to see how our algorithm performs regarding different aspects.
          
          \section{Preliminaries}\label{sec:setup}
          
          In this section, we collect some definitions and notations related to SDEs. Moreover, the  notion of weak convergence is introduced in detail, which lays the foundation of our construction of Gaussian mixture methods in later sections.
          
          \subsection{Notations and assumptions.}\label{subsec:notations}

          For the convenience of later discussions, we introduce for any integer $k$ the
          following set of functions
          \begin{equation}
          	C_b^k=\Big\{f\in C^k: \|f\|_{C^k}:=\sup_{x\in \mathbb{R}^d} \sum_{|\alpha|\le k} |D^{\alpha}f|<\infty \Big\}.
          \end{equation}
          Here the subscript $b$ is used to remind the reader that the functions are
          bounded in value and all their derivatives up to the specified order. We use $\mathbb{E}_x$ to denote the expectation under the law of the process $X(t)$ with $X(0)=x$. The notation $\mathcal{N}(m, \L)$ denotes the normal distribution with mean $m$ and covariance matrix $\L$. 
          
          We list out the following assumptions, which will be used throughout this work:
          \smallskip
          
\begin{assumption}\label{ass:pd}
The diffusion matrix
\begin{equation}
\L(x) :=\sigma(x)\sigma^T(x)
\end{equation}
is uniformly positive definite. In other words, 
\[
\inf_{x\in \mathbb{R}^d} \min\lambda(\L(x))\ge \sigma_0^2>0
\]
for some $\sigma_0>0$, where $\lambda(\L)$ is the set of eigenvalues of $\L$.
\end{assumption}
          
\medskip
          
\begin{assumption}\label{ass:bounded}
          	The coefficients are smooth in the sense that $b, \sigma\in C_b^6$. 
\end{assumption}
\medskip
      
          Note that Assumption~\ref{ass:bounded} implies that the coefficients are Lipschitz continuous, i.e. there exists a constant $K>0$ such
          that for all $x, y\in \mathbb{R}^d$,  \begin{equation}
          	|b(x)-b(y)|+|\sigma(x)-\sigma(y)|\le K|x-y|.
          \end{equation}
          It is well known that Assumption~\ref{ass:bounded} ensures the
          existence of strong solutions to \eqref{eq:sde} \cite{oksendal03} and
          that the moments of the solution are bounded:
          \[
          \sup_{0\le t \le T}\mathbb{E}_x|X(t)|^{2m}\le C(x,T).
          \]
Though likely overly restrictive,  Assumption~\ref{ass:pd} and Assumption~\ref{ass:bounded} will simplify the analysis and make the ideas more transparent. Analysis based on Assumption~\ref{ass:bounded} has been pursued in many works (see for example \cite{am11}). Compared with Assumption~\ref{ass:bounded}, some authors relax the coefficients to have polynomial growth (see for example \cite{milstein86}). The current results can be extended to locally Lipschitz coefficients with polynomial growth  under appropriate one-sided Lyapunov conditions or simply the arbitrary moment bounds they imply, see for instance \cite{mattinglyStuartHigham2002}.

          The generator of the diffusion process \eqref{eq:sde} is given by 
          \begin{equation}\label{eqn:actualgenerator}
          	\cL=\sum_{i=1}^d b_i\partial_i+\sum_{1\le i, j\le d} \frac{1}{2}\L_{ij}\partial_{ij}.
          \end{equation}
          The evolution of the law satisfies the Fokker-Planck equation (or the forward Kolmogrov equation)
          \begin{equation}\label{eq:fp}
          	\partial_tp=\cL^*p:=-\nabla\cdot(b p)+\frac{1}{2}\sum_{ij}\partial_{ij}^2(\L_{ij} p).
          \end{equation}
          For a smooth function $\phi$, let us define
          \begin{equation}
          	u(x, t)=\mathbb{E}_x\phi(X(t)).
          \end{equation}
          With regularity Assumptions~\ref{ass:pd} and~\ref{ass:bounded},  $u$ satisfies the backward Kolmogorov equation (see, for example, \cite[Chap. 8]{oksendal03})
          \begin{equation}\label{eq:backwardequation}
          	u_t=\cL  u.
          \end{equation}
          Formally, this implies the semigroup expansion
          \begin{equation}\label{eq:semigroupexpansion}
          	u(x, t)=e^{t\cL}\phi(x)=\sum_{j=0}^{\infty}\frac{t^j}{j!}\cL^j\phi(x).
          \end{equation}
          Given regularity assumptions on $\phi$, the expansion can be rigorously established up to a certain order. We cite a classical result in \cite[Chap. XI]{hp96} for expansion up to $j=2$, which has been modified for our purpose.
          \medskip
          
          \begin{lemma}{\cite[Theorem 11.6.4]{hp96}}\label{lmm:groupexp} 
          	Under Assumptions~\ref{ass:pd}--\ref{ass:bounded}, 
          	there exists a non-negative non-decreasing function $\rho$, such
          	that for all $\phi\in C_b^{\infty}$,
          	\begin{equation}
          		\sup_{x\in\mathbb{R}^d}\Big|u(x, h)-\Big(\phi(x)+\sum_{j=1}^{2}\frac{h^j}{j!}\cL^j\phi \Big)\Big| \le \rho(\|\phi\|_{C^6})h^3.
          	\end{equation}
          \end{lemma}

          In a numerical scheme, we generate the approximation sequences for the diffusion process at discrete time steps. Let $T>0$ be the terminal time point, and $N$ the number of numerical steps such that
          \begin{equation}
          	h=T/N.
          \end{equation} 
          We use $t_n=nh$ ($n=0,1, \ldots$) to denote the time grid points, $X^{n}$ the random variable generated by some numerical method to approximate $X(t_n)$, and $x_n$ a particular realization of the random variable $X^n$.

          \subsection{Weak convergence.}\label{subsec:weakconv}
          We only require the law of $X^n$ to approximate the law of the solution to (\ref{eq:sde}). This is described by the notion of weak convergence, which will be the focus of this section.
          
          \medskip
          
          \begin{definition}\label{def:weakcvg}
          	Fix $T>0$. Let $N, h$ and $X^n$ be given as in Section~\ref{subsec:notations}. We say $X^n$ converges weakly with order $r>0$ to $X(t_n)$ as $h\to 0$ if for any $\phi\in C_{b}^{\infty}$, there exist $C>0$, $h_0>0$ that are independent of $h$ (but may depend on $T$ and $\phi$) such that
          	\begin{equation}
          		\bigl\lvert \mathbb{E}\phi(X^n)-\mathbb{E}\phi(X(t_n))\bigr\rvert\le C h^r, \quad \forall\, 1\le n\le N, \mbox{ whenever }h\in (0, h_0).
          	\end{equation}
          	Here, $\mathbb{E}$ represents the expectation under the law of $X^n$ or $X(t_n)$.
          \end{definition}
      \medskip

 \begin{remark}
Note that the test functions here have bounded derivatives, and those used in \cite[Sec. 9.7]{kp92} and \cite{milstein86} have derivatives with polynomial growth. Test functions with bounded derivatives induce weaker topology but are much easier to handle (see e.g., \cite{am11}). The results can be extended to the more general setting with additional  work and assumptions to ensure boundedness of moments.
\end{remark}
          
          \smallskip
          
          We now move on to the criteria of weak convergence. Suppose the random sequence $X^n$ is generated by 
          \begin{equation}\label{eq:scheme}
          	X^{n+1}=X^{n}+A(X^n, \zeta^n, h), \, X^0=x,
          \end{equation}
          where $\zeta^n$ is a random vector generated at time $t_n$ and $A$ is a function. If $\{\zeta^n\}$'s are i.i.d, then  $\{X^n\}$ is a time-homogeneous Markov chain. (For example, in Algorithm~\ref{alg:1dode} below, $\zeta^n$ is a combination of the $z$ random variable and the standard 1D normal variable $\xi$.)
          
The following proposition is standard and we provide the proof in Appendix~\ref{app:globalerr} for reference. We emphasize that the following two results are not new and are part of the standard ``folklore'' meta-theorems in the subject. We repeat them only to make precise the versions we need and their dependences on parameters.

\smallskip

\begin{proposition}\label{pro:globalerror}
Let $b$ and $\sigma$ satisfy $b, \sigma\in C_b^{2(r+1)}$ for $r>0$. If there is a nonnegative and non-decreasing function $\rho$ such that for all $\phi\in C_b^{\infty}$, we have the local truncation error bounded by 
\[
|\mathbb{E}_x\phi(X(h))-\mathbb{E}_x\phi(X^1)|\le \rho(\|\phi\|_{C^{2(r+1)}})  h^{r+1}, \, \forall x\in\mathbb{R}^d,
\]
then $X^n$ converges weakly with order $r$ to $X(t_n)$ as $h\to 0$.
\end{proposition}

\smallskip
          
As before, $\mathbb{E}_x$ represents the expectation under the law of the process or Markov chain starting at $x$. This proposition basically says that if the local truncation error is $O(h^{r+1})$, then the global error is of order $r$. We have the following trivial observation by Lemma~\ref{lmm:groupexp}
and Proposition \ref{pro:globalerror}:

\smallskip

\begin{corollary}\label{cor:second}
Under Assumptions~\ref{ass:pd}-\ref{ass:bounded}, if there exists $\rho$ that is nonnegative and non-decreasing such that $\forall \phi\in C_b^{\infty}$, we have
          	\begin{equation}\label{eq:basic2ndorder}
          		\sup_{x\in \mathbb{R}^d}\Bigl\lvert \mathbb{E}_x(\phi(X^1))-\sum_{j=0}^{2}\frac{h^j}{j!}\cL^j\phi(x)   \Bigr\rvert\le \rho(\|\phi\|_{C^6})h^{3},
          	\end{equation}
          	then the method \eqref{eq:scheme} is of weak second-order accuracy.
\end{corollary}

          \section{Weak second-order Gaussian mixture method}\label{sec:gauss}
          
          The Euler-Maruyama scheme \cite{kp92} for SDE \eqref{eq:sde}
          \begin{equation}\label{eq:em}
          	X^{n+1}=X^n+b(X^n)h+\sigma(X^n)\Delta W_n
          \end{equation} 
          generates Gaussian distributions for $X^{n+1}$ conditioning on $X^n=x_n$ but has only weak first-order accuracy. It is well known that constructing a weak second-order scheme is nontrivial, not to mention a weak second-order scheme using Gaussian approximations for the measure $\mu(X^{n+1}|X^n=x_n)$. In fact, as we will see, an approximation using a single Gaussian is generally insufficient for weak second-order accuracy. Hence, we aim to use Gaussian mixture to construct higher order schemes.

          To start with, let us recall that the law of  $Y(t)$, the weak solution of the following SDE with additive noise, is a Gaussian distribution if $Y(0)$ is a Gaussian random variable independent of $W$:
          \begin{equation}
          	dY(t)=\mu(t)dt+\sigma(t)dW.
          \end{equation} 
Here, $\mu(t)$ and $\sigma(t)$ only depend on time. The mean and covariance matrix of $Y(t)$ are given respectively by
          \begin{equation}
          	m(t)=m_0+\int_0^t\mu(s)\,ds,~~S(t)=S_0+\int_0^t\sigma(s)\sigma^T(s)\,ds.
          \end{equation}
          Conversely, if we are given some $m$ and some positive semidefinite matrix $S$, we can recover the normal distribution $\mathcal{N}(m, S)$ by constructing ``paths" $\{m(t)\}_{t\in [0,h]}$ and  $\{S(t)\}_{t\in [0,h]}$ with $m(h)=m,\, S(h)=S$, and $\dot{S}$ being positive semi-definite so that the solution to the SDE
          	\[
          	dY(t)=\dot{m}(t)\,dt+\sqrt{\dot{S}(t)}\,dW
          	\]
          	with $Y(0)\sim \mathcal{N}(m(0), S(0))$, independent of $W$, will satisfy $Y(h) \sim \mathcal{N}(m,S)$. Here $\dot{m}$ and $\dot{S}$
          	denote their respective time derivatives. Now let us consider the time-dependent generator 
          	\begin{equation}
          		\cL(s)= \dot{m}(s) \cdot \nabla_x + \dfrac{1}{2}\dot{S}_{ij}(s)\partial_{ij} .
          	\end{equation}
          	We now assume $S(0)=0$ and $m(0)=x_0$ for some $x_0$. By the backward equation \eqref{eq:backwardequation}, we have
          	\begin{equation}
          		\mathbb{E}\phi(Y(h))=\exp\Big(\int_0^h\cL(s)\,ds\Big)\phi(x_0)=\exp\Big((m-x_0)\nabla_x+\frac{1}{2}S_{ij}\partial_{ij}\Big)\phi(x_0).
          	\end{equation}
          	Here the second equality comes from integrating $\mathcal{L}$ in time. Therefore, for any random variable $Y\sim \mathcal{N}(m, S)$, we can express the expectation of $\phi$ as
	\begin{equation}\label{eq:keyidentity}
\mathbb{E}\phi (Y)=\exp(\cL_z)\varphi(z),~\forall z\in \mathbb{R}^d,
\end{equation}
where
          	\begin{equation}
          		\cL_z=(m-z)\partial_x+\frac{1}{2}S_{ij}\partial_{ij}.
          \end{equation}
          We will use \eqref{eq:keyidentity} to construct our scheme. In this section, we start with $d=1$. Our construction for $d=1$ will be used as the building block for constructing our scheme in higher dimensions in Section~\ref{sec:multidim}.
          
\subsection{Conditions for second order Gaussian mixtures.}
          
          First of all, we claim that using a single Gaussian distribution to approximate $\mu(X^{n+1}| X^n=x_n)$ is generally insufficient for weak second order accuracy. To start with, we assume $X^1$ generated by \eqref{eq:scheme} conditioning on $X^0=x_0$ is a normal distribution with mean $m(h,x_0)$ and variance $S(h,x_0)$:
          \begin{equation}\label{eq:markovchain}
          	X^1=x_0+A(x_0, \xi, h) \sim \tilde{\rho}(x; h, x_0)=\frac{1}{\sqrt{2\pi S(h,x_0)}}\exp\Big(-\frac{(x-m(h,x_0))^2}{2S(h,x_0)}\Big).
          \end{equation}
          Here $X^1\sim \tilde{\rho}$ means the law of $X^1$ has a density $\tilde{\rho}$. Using \eqref{eq:keyidentity}, we desire 
 \begin{equation}
          	\exp\Big((m(h,x_0)-x_0)\partial_x+\frac{1}{2}S(h,x_0)\partial_{xx}\Big)\varphi(x_0)
          	=\varphi(x_0)+h\cL\varphi(x_0)+\frac{h^2}{2}\cL^2\varphi(x_0)+O(h^3)
 \end{equation}
          in order to achieve global weak second-order accuracy. Clearly, we need $m(h,x_0)-x_0=o(1), S(h,x_0)=o(1)$ as $h\to 0$. Using the semigroup expansion (Lemma~\ref{lmm:groupexp}), we infer that 
          	\begin{align*}
          		m(h,x_0)&=x_0+m_1(x_0)h+\frac{1}{2}m_2(x_0)h^2+R_1(h,x_0) h^3,\\
          		S(h,x_0) &=S_0(x_0)h+\frac{1}{2}S_1(x_0)h^2+R_2(h,x_0)h^3,
          	\end{align*}
          where $R_1, R_2$ are bounded. Detailed calculation shows the following:
          
          \medskip
          \begin{proposition}\label{pro:firstorderofonegauss}
          	For a general multiplicative noise (or equivalently $\sigma(x)$ is not constant), there exist no $(m_0, m_1, m_2, S_0, S_1)$ as functions of $x_0$ such that the constraint \eqref{eq:basic2ndorder} can be satisfied. 
 \end{proposition}

\smallskip

          \begin{remark}
          	By the proof of Proposition \eqref{pro:firstorderofonegauss} in Appendix~\ref{app:firstgaussian}, if the noise is additive ($\sigma$ is independent of $x$), it is possible to construct an approximation with a single Gaussian that yields global weak second-order accuracy.
          \end{remark}
      
      \smallskip
          
The proof of Proposition \ref{pro:firstorderofonegauss} is provided in Appendix~\ref{app:firstgaussian}. Proposition \ref{pro:firstorderofonegauss} is a strong indication that no approximation with one Gaussian can reach weak second-order accuracy, which forces us to seek Gaussian mixtures. In the derivation below, we use $R(x)$ to denote a generic function with a bound that depends only on $\|\cdot\|_{C^6}$ norms of $b, \sigma$, and $\phi$ (the test function), and its concrete meaning may change from line to line. Below, we would first present an informal argument to derive the scheme; the rigorous analysis of the scheme will be deferred to later sections.

          As we have mentioned, considering the law of $X^1$ given the initial position  $X^0=X(0)=x_0$ is sufficient to determine the whole Markov chain by time homogeneity. Therefore, it suffices to consider that the law of $X^1$ is given by a mixture of $M$ Gaussians:
          \begin{equation}
          	X^1\sim \sum_{i=1}^M w_i \mathcal{N}(m_i(h), S_i(h)).
          \end{equation}
          Here we abuse notation by letting $\mathcal{N}(m, S)$ denote the
          density function of a Gaussian with the given mean and covariance. 
          
          Let $\cL_i:=(m_i(h)-x_0)\partial_x+\frac{1}{2}S_i(h)\partial_{xx}$. By \eqref{eq:keyidentity}, we have
          \begin{equation}
          	\mathbb{E}\phi(X^1)=\sum_{i=1}^M w_i \exp(\cL_i)\phi(x_0).
          \end{equation}
          Here, the dependence on $x_0$ in the coefficients is not written out explicitly for simplicity. Since after time $h$,   the scale for a pure
          diffusion process is $\sqrt{h}$ (since
          $\mathbb{E}|X(t+h)-X(t)|^2\sim h$), we expect that $|m_i(h)-x_0|\le C\sqrt{h}$ and $|S_i(h)|\le C_Mh$. Therefore, by Corollary~\ref{cor:second}, the scheme will be of weak second-order if the following holds for all $\phi\in C_b^{\infty}$:
          \begin{multline}\label{eq:2ndcondition1}
          	\sum_{i=1}^M w_i\Bigl(\cL_i\phi(x_0)+\frac{1}{2}\cL_i^2\phi(x_0)+\frac{1}{6}\cL_i^3\phi(x_0)+\frac{1}{24}\cL_i^4\phi(x_0)\Bigr)\\
          	-\Bigl(h\cL\phi(x_0)+\frac{1}{2}h^2\cL^2\phi(x_0)\Bigr)=R(x_0,h)h^3,
          \end{multline}
          where $R(h)$ is bounded and depends on the function $\phi$. We stop at $\cL_i^4$ because of the expectation $|m_i(h)-x_0|\le C\sqrt{h}$ so $\cL_i^k$ is of order at least $O(h^\frac{k}{2})$. Note that by \eqref{eqn:actualgenerator},
          	\begin{align}
          		\cL\phi(x_0)&=b(x_0)\phi'(x_0)+\frac{1}{2}\L(x_0)\phi''(x_0), \nonumber \\
          		\frac{1}{2}\cL^2\phi(x_0) & =\frac{1}{2}\Big(b(x_0)(b\phi'+\frac{1}{2}\L\phi'')'+\frac{1}{2}\L(x_0)(b\phi'+\frac{1}{2}\L \phi'')'' \Big).\label{eq:constraint1d}
          	\end{align}
          Due to the $\sqrt{h}$ scale in displacement, we take the following ansatz for $m_i(h)$ and $S_i(h)$:
          		\begin{align}
          			m_i(h)&=x_0+m_{i0}h^{1/2}+m_{i1} h+m_{i2} h^{3/2}+m_{i3} h^2+m_{i4} h^{5/2}+R_{i1}(h)h^3, \nonumber \\
          			S_i(h)&=S_{i1} h+S_{i2} h^{3/2}+S_{i3} h^2+S_{i4} h^{5/2}+R_{i2}(h)h^3 > 0.\label{eq:mSexpand} 
          		\end{align}
          	Substituting the ansatz \eqref{eq:mSexpand} into \eqref{eq:2ndcondition1},  after a tedious but straightforward calculation, we are able to derive the following conditions:
          		\begin{align}
          			&\sum_{i}w_i m_{i1}=b, \nonumber\\
          			&\frac{1}{2}\sum_{i}w_i S_{i1}+\frac{1}{2}\sum_i w_i m_{i0}^2=\frac{1}{2}\L, \nonumber \\
          			&\sum_i w_im_{i3}=\frac{1}{2}bb'+\frac{1}{4}\L b'', \nonumber \\
          			&\frac{1}{2}\sum_i w_i S_{i3}+\frac{1}{2}\sum_i w_i(2m_{i0}m_{i2}+m_{i1}^2)=\frac{1}{2}b^2+\frac{1}{4}b\L'+\frac{1}{2}\L b'+\frac{1}{8}\L\L'', \nonumber \\
          			&\frac{1}{2}\sum_i w_im_{i1}S_{i1}+\frac{1}{2}\sum_i w_i m_{i0}S_{i2}
          			+\frac{1}{2}\sum_i w_i m_{i0}^2m_{i1}
          			=\frac{1}{2}b\L+\frac{1}{4}\L\L', \nonumber \\
          			&\frac{1}{8}\sum_i w_iS_{i1}^2+\frac{1}{4}\sum_i w_i m_{i0}^2S_{i1}
          			+\frac{1}{24}\sum_i w_i m_{i0}^4=\frac{1}{8}\L^2, \nonumber \\ & \mbox{All the odd powers of } h^{1/2} \mbox{ vanish in }\sum_i w_i \cL_i^m \mbox{ for any } m=1,2,3,4. \label{eq:2ndCond2}
          		\end{align}
          	In the above equations, functions $b,\L$ and their spatial derivatives are all evaluated at point $x_0$.
          
          In the following Section \ref{subsec:1dode}, we consider a possible approach to satisfy these
          constraints, by choosing $M=3$.
          
          \medskip
          
          \begin{remark}
          	We have not yet derived a weak third order Gaussian
          	mixture scheme. The number of  variables and the equations
          	grow to the point where our current methods to solve them are
          	unfeasible. However, we expect that a minimum of five Gaussians is needed to reach third order, which is suggested by the second and sixth equations of (\ref{eq:2ndCond2}). These are constraints for $\phi''$ in first order and $\phi^{(4)}$ in second order respectively (and there will be another constraint for $\phi^{(6)}$ in third order), which only involve the weights $w_i$ and the leading order diffusion scaling terms $m_{i0}$ and $S_{i1}$.
          \end{remark}
          \subsection{An ODE approach.}\label{subsec:1dode}
          In this section, we give a particular construction of $m_i(h)$ and $S_i(h)$ that satisfy \eqref{eq:2ndCond2} which determines our numerical scheme. Our approach is to construct ODEs for $m_i$ and $S_i$ with a certain initial condition and solve them at time $h$, which has the advantage of avoiding derivative evaluations of $b$ and $\L$.
          
          To satisfy the last condition of \eqref{eq:2ndCond2}, we consider
          	a ``symmetric'' construction. It is convenient to relabel the
          	Gaussians as $i=0,\pm 1$, so that $\mathcal{N}(m_0(h),S_0(h))$ is
          	``centralized" and does not contribute to the odd powers of
          	$h^{\frac{1}{2}}$. The centers of the other two Gaussians
          	$m_{\pm 1}(h)$ are placed at both sides of $m_0(h)$ with
          	$O(\sqrt{h})$ distance apart, and the variance matrices
          	$S_{1}(h), S_{-1}(h)$ are constructed similarly with each
          	other. These two Gaussians will contribute powers like
          	$h^{k/2}$. Moreover, we impose $w_1=w_{-1}$, so that the odd powers of $h^\frac{1}{2}$
          	will cancel out with each other due to symmetry.
          
          For initial conditions, we set 
          		\begin{align}
          			& m_i(0)=x_0+z_i\sqrt{\gamma h\L(x_0)}, \, S_i(0)=0, \nonumber \\
          			& w_1=w_{-1}, ~~w_0+2w_1=1,\label{eq:initcondmi0}
          		\end{align}
          	where $\gamma>0$ is a parameter and 
          	\begin{equation}
          		z_i=i, ~i=0,\pm 1.
          	\end{equation}
          	This choice takes into consideration that the diffusion scale is $\sqrt{h}$, while transportation scale is $h$. For the choice of ODE flows, we take the following ansatz, where the functions $g_i$ are to be determined later:
          	\begin{subequations}\label{1dodeflow}
          		\begin{align}
          			& \dot{m}_i(t)=b(m_i(t)), \label{1dodemi}\\
          			& \dot{S}_i(t)=g_i(m_i(t)). \label{1dodeSi}
          		\end{align}
          	\end{subequations}
          	Our choice in \eqref{1dodemi} is natural in the sense that we expect 
          	\begin{equation*}
          		\dfrac{d}{dt} \mathbb{E} X(t) = \mathbb{E} b(X(t)) \approx b(\mathbb{E} X(t)),
          	\end{equation*} 
          	and the approximation is exact if $b$ is a linear function. For symmetry, we require $g_1=g_{-1}$. Clearly, the $\sqrt{h}$ factor enters $m_i(h)$ through the initial value and then the equation. Due to symmetries in both $m_{\pm 1}$ and $S_{\pm 1}$ and $w_1=w_{-1}$, all the odd powers of $h^{1/2}$ indeed cancel out in $\sum_i w_i \cL_i^m$.

          We now find the constraints on the functions $g_i$ and the parameters so that \eqref{eq:2ndCond2} can be satisfied.
          To start with, we have by Taylor expansion that
\begin{equation}\label{eq:expansionmh}
          	m_i(h)=m_i(0)+b(m_i(0))h+\frac{1}{2}b'(m_i(0))b(m_i(0))h^2+O(h^3).
\end{equation}
          	Hence, substituting \eqref{eq:initcondmi0} into \eqref{eq:expansionmh}, and considering our ansatz \eqref{eq:mSexpand}, we obtain
          	\begin{equation}\label{eqn:mi}
          		m_{i0}=z_i\sqrt{\gamma\L},~m_{i1}=b(x_0),~m_{i2}=z_ib'\sqrt{\gamma\L},~m_{i3}=\frac{1}{2}b''z_i^2\gamma\L+\frac{1}{2}bb'.
          	\end{equation}
          	Similarly, we can find $S_{ij}$'s:
          	\begin{equation}\label{eq:Sij}
          		S_{i1}=g_i(x_0), \, S_{i2}=z_ig_i'\sqrt{\gamma\L},~
          		S_{i3}=\frac{1}{2}z_i^2g_i''\gamma\L+\frac{1}{2}g_i'b.
          	\end{equation}
          	However, substituting \eqref{eqn:mi} and \eqref{eq:Sij} into \eqref{eq:2ndCond2} cannot uniquely determine the parameters $w_i,~g_i$ and $\gamma$. We further impose $S_{01}=S_{11}$ which makes our construction of the scheme easier for higher dimensions, which uniquely determines the solution:
          		\begin{align}
          			& \gamma=\frac{3}{2},~w_1=\frac{1}{6},~w_0=\frac{2}{3}, \nonumber \\
          			& g_0(x_0)=g_1(x_0)=\frac{1}{2}\L(x_0), \nonumber \\
          			& g_0'(x_0)=g_1'(x_0)=\L'(x_0), \nonumber \\
          			& g_1''(x_0)=\L''(x_0).\label{eq:theconstraints}
          		\end{align}
          	Clearly, choosing the following functions will suffice:
          	\begin{equation}\label{eq:funcg}
          		g_0(x)=g_1(x)=g(x):=\L(x)-\frac{1}{2}\L(x_0).
          \end{equation}
          Unfortunately, this choice has one issue: $g_0$ and $g_1$ are not always
          nonnegative. Indeed, it is possible that $S_i(h)$ given by \eqref{1dodeSi}
          could be negative. To solve this issue,
          we simply set $S_i(h)$ to zero if that happens. Fortunately, since $g(x) \approx \frac{1}{2}\L(x_0)$ is positive whenever $x$ is close to $x_0$, $S_i(h)$ can be guaranteed to be positive whenever $h$ is sufficiently small, thus it can be shown that this error has a
          lower-order effect. Similar situation also arises in \cite{am11}.

          The details of the procedure outlined above are
          expressed more exactly in the following Algorithm \ref{alg:1dode} which gives the
          pseudocode to generate $x_{n+1}$ from 
          $x_n$.
          
          \begin{algorithm}
          	\caption{Gaussian mixture scheme for SDEs (ODE method in 1D)\label{alg:1dode}}
          	\begin{algorithmic}[1]
          		\State Generate $z$ such that $P(z=0)=\frac{2}{3}$
          		and $P(z=1)=P(z=-1)=\frac{1}{6}$. Then, set
          		\begin{equation}
          			m(0)=x_n+z\sqrt{\frac{3}{2} \L(x_n) h}.
          		\end{equation}
          		
          		\State Solve the ODEs 
          			\begin{align}
          				& \dot{m}=b(m), \nonumber \\
          				& \dot{S}=g(m(t)),\label{eqn:1dode}
          			\end{align}
          		using an ODE solver of at least second-order accuracy (for example, Runga-Kutta methods of order $k\ge 2$) to obtain $m(h)$, $S(h)$. Here $g(x)=\L(x)-\frac{1}{2}\L(x_n)$.
          		
          		\State If $S(h)\le 0$, then $x_{n+1}=m(h)$.
          		If $S(h)>0$, then 
          		\begin{equation}
          			x_{n+1}=m(h)+\sqrt{S(h)}\xi,
          		\end{equation}
          		where $\xi$ is a standard 1D normal variable.
          	\end{algorithmic}
          \end{algorithm}
          
          \begin{remark}
          	One may truncate the function and  consider 
          	\[
          	g_0(x)=g_1(x)=\psi(x; x_0)(\L(x)-\L(x_0))+\frac{1}{2}\L(x_0)
          	\]
          	where $\psi(x; x_0)$ is some truncation function that is $1$ in a neighborhood of $x_0$ so that $g_0, g_1$ are positive definite for all $x$. This approach, however, is not very convenient and in practice the behavior is not very satisfactory. 
          \end{remark}
          
          We are now in position to present the following theorem, which tells that our scheme is indeed of weak second-order.
          
          \medskip
          
          \begin{theorem}\label{thm:1dodeflow}
          	Let $d=1$. Suppose Assumptions~\ref{ass:pd}-\ref{ass:bounded} hold., then Algorithm~\ref{alg:1dode} is a weak second-order scheme for SDE \eqref{eq:sde}.
          \end{theorem}
          \begin{proof}
          	It is clear that there exists $h_0>0$ such that for $h<h_0$,
          	\[
          	\sqrt{\tfrac{3}{2}\|\L\|_{\infty}h}+\|b\|_{\infty}h
          	<\frac{\sigma_0^2}{2\|\L'\|_{\infty}}.
          	\]
          	Consider that $X(0)=X^0=x_0$. By construction, $|m_i(t)-m_i(0)|\le \|b\|_{\infty}h$ for all $t\le h$, and we have
          	\[
          	g(m_i(t))\ge \frac{1}{2}\L(x_0)-\|\L'\|_{\infty}\Big(|m_i(t)-m_i(0)|+|z_i|\sqrt{\frac{3}{2} \|\L\|_{\infty} h} \Big)>0.
          	\]
          	Hence, $S_i(h)>0$ for $h<h_0$. Moreover, any reasonable numerical approximation to
          	$S_i(h)$ will also be positive for sufficiently small $h$.
          	
          	By \eqref{eq:funcg}, we can conclude that for $h<h_0$, \eqref{eq:2ndCond2} holds, and $S_i(h)>0$.
          	In other words, \eqref{eq:2ndcondition1} holds and 
          	\[
          	\Big|\mathbb{E}_{x_0}(\phi(X^{1}))-(\phi(x_0)+\cL\phi(x_0)+\frac{1}{2}\cL^2\phi(x_0))\Big|
          	\le \rho(\|\phi\|_{C^6})h^3.
          	\]
          	By Corollary~\ref{cor:second}, we find that our scheme constructed here is of weak second-order if \eqref{eqn:1dode} is solved exactly. 
          	Since for any numerical solver on \eqref{eqn:1dode} that is of at least second order, the error induced by solving \eqref{eqn:1dode} is $O(h^3)$ or smaller, and therefore the above local estimate still holds. Our Algorithm \ref{alg:1dode} is thus of weak second order as well.
          \end{proof}
      
      \medskip
          
          \begin{remark}
          	The above construction with ODE flow gives $S_i(h)$ that can be
          	possibly negative, though it is positive asymptotically as $h \to 0$
          	and when it becomes negative, we can always fix by setting it to
          	zero. One may desire to have a method that ensures $S_i(h)$ to be
          	positive. In Appendix \ref{app:varconstruct}, we provide a direct way to construct $S_i$'s so that positivity can be guaranteed. However, this method involves evaluation of the derivatives of $\Lambda$, which is oftentimes undesired.
          \end{remark}
          
          \section{Gaussian mixture for  multi-dimensions}\label{sec:multidim}
          
          In this section, we generalize the Gaussian mixture method
          constructed in Section~\ref{sec:gauss} to higher dimensions. We assume
          that we have the eigen-decomposition for $\L(x)$:
          \begin{equation}\label{eq:Lambdadecomp}
          	\L(x)=\sum_{i=1}^d \lambda_i(x) v_i(x)v_i^T(x),
          \end{equation}
          where $\lambda_i(x)$'s are the
          eigenvalues of the matrix $\L(x)$, and $\{v_i\}$ forms an orthonormal basis of $\mathbb{R}^d$.

          As discussed in Section~\ref{sec:gauss}, we only need to focus on how to generate $X^1$ given $X^0=x_0$. 
          Again, we assume that $X^1$ has the conditional probability measure of the form 
          \begin{equation}\label{eq:multdnormalass}
          	\bar{\rho}=\sum_{p\in P} w_p \mathcal{N}(m_p(h), S_p(h)),
          \end{equation}
          for Gaussian mixture approximations. Here we use $P$ to denote the set of indices $p$.
          
          To illustrate our choice of the number of Gaussians and their initial positions, suppose we have $d$-dimensional decoupled diffusion process (diffusion matrix is diagonal), then we approximate each dimension using our 1D technique in Section~\ref{sec:gauss} and then get a global second-order approximation. In each dimension, we have three Gaussians, which means we have a total of $3^d$ Gaussians. If the diffusion matrix is no longer diagonal, we can still consider using $3^d$ Gaussians. 
          At the first glance, the complexity is large, but fortunately, it turns out that the complexity grows linearly with $d$ instead of exponentially. 
          
          We now explain our construction. Let the index set $P=\{-1,0,1\}^d$, so that $|P|=3^d$, and each index $p\in P$ can be expressed as $p=(z_p^1,\cdots,z_p^d)$ where $z_p^i\in \{0,\pm 1\}$. Let us consider the Gaussians with initial centers $y_p$,  given by
          \begin{equation}
          	y_p=x_0+\sum_{i=1}^d z_p^i \sqrt{\gamma \lambda_i h} v_i , \, \mbox{ where } \gamma=\frac{3}{2}.
          \end{equation}
          These formulas and  $\gamma=\frac{3}{2}$ are obtained from the 1D construction in Section~\ref{sec:gauss}.  The weight for the Gaussian with index $p$ is 
          \begin{equation}\label{eq:weightsmd}
          	w_p=\prod_{i=1}^d w^{z_p^i},~~1\le p\le 3^d,
          \end{equation}
          with the parameters given by
          \begin{equation}
          	w^1=w^{-1}=\frac{1}{6},~w^0=\frac{2}{3}.
          \end{equation}
          
          \begin{remark}
          	Another natural idea is to place the initial points at $x_0, ~x_0\pm \sqrt{\gamma \lambda_i h}v_i$ and there are $2d+1$ such points. After some attempts, we found that this strategy hardly works when $d$ is large. 
          \end{remark}
      
      \smallskip
          
          With these initial positions and weights, we can easily generalize our Gaussian mixture constructions for $d=1$ to arbitrary dimensions. 
          
          \subsection{The ODE approach for multi-dimensions.}
          
          Following the construction in the 1D case, we consider $m_p(h)$ and $S_p(h)$ for $p\in P$ given by
          	\begin{align}
          		& \dot{m}_p(t)=b(m_p(t)), ~ m_p(0)=y_p, \nonumber \\
          		& \dot{S}_p(t)=G(m_p(t)),~ S_p(0)=0,\label{eq:mdodesys}
          	\end{align}
          where
          \begin{equation}\label{eq:Gx}
          	G(x)=\L(x)-\frac{1}{2}\L(x_0).
          \end{equation}
          Thanks to imposing $S_{01}=S_{11}$ in (\ref{eq:Sij}) we are able to have a simple expression (\ref{eq:Gx}). The algorithm can then be summarized as the following Algorithm~\ref{alg:multdode}.
          \begin{algorithm}[!hbp]
          	\caption{Gaussian mixture methods for SDEs (ODE method in higher D)\label{alg:multdode}}
          	\begin{algorithmic}[1]
          		\State Compute the matrix eigen-decomposition 
          		\begin{equation}
          			\L(x_n)=\sum_{i=1}^d \lambda_i v_iv_i^T.
          		\end{equation}
          		
          		\State Generate $z^i, i=1,2,\ldots, d$ so that $P(z^i=0)=\frac{2}{3}$ while $P(z^i=1)=P(z^i=-1)=\frac{1}{6}$.
          		
          		\State Let
          		\begin{equation}\label{eqn:highdm0}
          			m(0)=x_n+\sum_{i=1}^d z^i \sqrt{\frac{3}{2} \lambda_i h} v_i,
          		\end{equation}
          		and find $m(h)$ by solving
          		\[
          		\dot{m}(t)=b(m(t))
          		\]
          		using a method with at least second-order accuracy.
          		
          		\State
          		Find $S(h)$ by solving
          		\[
          		\dot{S}(t)=G(m(t)), ~S(0)=0.
          		\]
          		
          		\State If $S(h)$ is not positive definite, then set all negative eigenvalues to zero (keeping the same eigenvectors) and obtain $\tilde{S}(h)$. Sample $x_{n+1}\sim \mathcal{N}(m(h), \tilde{S}(h))$. In other words,
          		\[
          		x_{n+1}=m(h)+\sum_{i=1}^d \sqrt{\mu_i^+}\xi_i u_i
          		\]
          		where $S(h)=\sum_{i=1}^d \mu_i u_i u_i^T$ with $\{u_i\}_{i=1}^d$ being orthonormal, $\mu_i^+=\max(\mu_i, 0)$ and $\{\xi_i\}$ are i.i.d standard 1D normal variables.
          	\end{algorithmic}
          \end{algorithm}
      
      \smallskip
      
          \begin{remark}
          	Our algorithm requires a matrix factorization at every time step, which is the most computationally costly step. However, as $\L(X(t))$ does not change much between consecutive time steps, one could use the matrix of $v_i$'s as the preconditioner for next step's computation, which will significantly reduce the computational cost in high dimensions.
          \end{remark}
      \smallskip
      
          We now establish the main result for multi-dimensions:
          
          \smallskip
          
          \begin{theorem}\label{thm:odemultid}
          	Suppose the Assumptions~\ref{ass:pd}-\ref{ass:bounded} are satisfied,
          	then there exists $h_0>0$ such that when $h<h_0$:
          	
          	(i)  $S_p(h)$ is positive definite for all $p\in P$  and for any initial position $x_0$.
          	
          	(ii) for any test function $\phi\in C_b^{\infty}$, there exists a constant $C$ depending on the $C^6$ norms of $\phi, b, \sigma$ only, such that
          	\[
          	|\mathbb{E}_{x_0}(\phi(X^{1}))-\mathbb{E}_{x_0}(\phi(X(t_{1})))|\le C h^3,~h<h_0.
          	\]
          	Consequently, the Gaussian mixture Algorithm~\ref{alg:multdode} is a weak second-order scheme to \eqref{eq:sde}.
          \end{theorem}

      \smallskip
          
          To prove this theorem, we first present a useful lemma, the proof of which is deferred to Appendix~\ref{app:proofusefullmm}:
          
          \smallskip
          \begin{lemma}\label{lmm:usefullmm}
          	For a function $\phi\in C_b^{\infty}$, we have
          	\begin{align*}
          		\sum_{p\in P}w_p \phi(y_p)
          		& =\phi(x_0)+\sum_{i=1}^d w^1 D_i^2\phi(x_0) \gamma \lambda_i h 
          		+\frac{1}{2}\sum_{i\neq j} (w^1)^2\gamma^2D_i^2D_j^2\phi(x_0)  \lambda_i\lambda_j h^2
          		\\ & \qquad +\frac{1}{12}\sum_{i=1}^d (w^1\gamma^2) D_i^4\phi(x_0) \lambda_i^2h^2
          		+R(h)h^3.
          	\end{align*}
          	Here we use shorthand notation $D_i:=D_{v_i}$.
          \end{lemma}
      
      \smallskip
          
          \begin{remark}
          	Here $D_{v_i}\phi(x):=v_i(x_0)\cdot\nabla\phi(x)$, so we have $D_{v_i}^2\phi=
          	v_i\cdot\nabla(v_i\cdot\nabla\phi(x))=v_i\otimes v_i:\nabla^2\phi(x)$. The function inside is $v_i(x_0)\cdot\nabla\phi(x)$. In other words, we allow $\phi$ to change for $x\neq x_0$ but $v_i$ is frozen to be its value at $x_0$.
          \end{remark}
      
      \smallskip
          
          \begin{proof}(\textbf{Proof of Theorem~\ref{thm:odemultid}})
          	
          	(i). To prove this claim, we find that for all $p \in P$,
          	\begin{equation}\label{eqn:mphdiffoh}
          		|m_p(t)-m_p(0)|\le \|b\|_{\infty}h, ~\forall t\in [0, h].
          	\end{equation}
          	Hence, 
          	\begin{align*}
          		\min \lambda(G(m_p(t))) & \ge  \frac{1}{2}\min \lambda(\L(x_0))-\|\L(m_p(t))-\L(x_0)\|_2 \\
          		& \ge \frac{1}{2}\min \lambda(\L(x_0))-\sup_{x\in\mathbb{R}^d}\|\L'\|_{2}\Big(|m_p(t)-m_p(0)|+\sqrt{\frac{3}{2}\max_{1\le i\le d}\lambda_i h}\Big).
          	\end{align*}
          	Recall that we use $\lambda(M)$ to represent the set of eigenvalues of matrix $M$.  If $h$ is sufficiently small, $\min \lambda(G(m_p(t))) $ is positive for all $p\in P$ for $t\le h$.
          	By Equation \eqref{eq:mdodesys}, $\min \lambda(S_p(h))$ is positive for all $p$.
          	
          	(ii). Noticing that $\partial_{ijkl}\phi$ is a symmetric tensor on any indices, we find (the Einstein summation convention is used)
          	\begin{align*}
          		\mathbb{E}_{x_0}(\phi(X^{1})) & =\sum_{p\in P}w_p\Big( \phi+\frac{1}{2}\partial_i\partial_j\phi S_{p,ij}(h)
          		+\frac{1}{8}\partial_{ijkl}\phi S_{p,ij}S_{p,kl} \Big)\Big|_{x=m_p(h)}+R(h)h^3\\
          		& =\sum_{p\in P}w_p\Big[\phi+b_i\partial_i\phi h
          		+\frac{1}{2}(b_i\phi_{ij}b_j+b_i(\partial_ib_j)\partial_j\phi) h^2
          		+\frac{1}{2}\partial_{ij}\phi G_{ij}h\\ & \qquad
          		+\frac{1}{4}\Big(2\partial_{ijk}\phi b_k G_{ij}+\partial_{ij}\phi\partial_k G_{ij}b_k\Big)h^2
          		+\frac{1}{8}\partial_{ijkl}\phi G_{ij}G_{kl}h^2\Big]\Big|_{x=m_p(0)} \\ & \qquad +R(h)h^3.
          	\end{align*}
          	Using Lemma~\ref{lmm:usefullmm}, we are able to compute the sums. For example, we find:
          	\begin{align*}
          		\sum_{p\in P} w_p \frac{1}{2}\partial_{ij}\phi G_{ij}h \Big|_{x=m_p(0)}
          		& =\frac{1}{4}h \lambda_m D_m^2\phi
          		+\frac{h^2}{4}\sum_{m\neq n} w^1 D_m^2D_n^2\phi \gamma\lambda_m \lambda_n 
          		\\ & \qquad +\frac{1}{2}h^2 w^1 \Big(D_m^2\L_{ij}  \partial_{ij}\phi+2D_m\L_{ij} D_m\partial_{ij}\phi\Big)  \gamma \lambda_m \\ & \qquad +\frac{1}{4}h^2 w^1 D_m^4\phi \gamma \lambda_m^2 +R(h)h^3.
          	\end{align*}
          	Here, we used \eqref{eq:Gx} and identities like
          	\[
          D_m^2\partial_{ij}\phi \frac{1}{2}\L_{ij}=\frac{1}{2}
          	D_m^2D_n^2\phi \lambda_n.
          	\]
          	Noting 
          	\[
          	\gamma=3/2,~w^1\gamma=1/4,
          	\]
          	we have after some computation:
          	\begin{equation}
          		\mathbb{E}_{x_0}(\phi(X^{1}))=\phi(x_0)+Ah+Bh^2+R(h)h^3,
          	\end{equation}
          	where
          	\[
          	A=b_i\partial_i\phi+\frac{1}{2}  \lambda_m D_m^2\phi,
          	\]
          	and
          	\begin{align*}
          		B & =\frac{1}{8}\sum_{m\neq n}  D_m^2D_n^2\phi \lambda_m\lambda_n 
          		+\frac{1}{8} D_{m}^4\phi \lambda_m^2+\frac{1}{4}D_m^2(b_i\partial_i\phi) \lambda_m
          		\\ & \qquad +\frac{1}{2}(b_i\partial_{ij}\phi b_j+b_i\partial_ib_j \partial_j\phi)  
          		+\frac{1}{8} (D_m^2\L_{ij} \partial_{ij}\phi+2D_m\L_{ij} D_m\partial_{ij}\phi) \lambda_m
          		\\ & \qquad+\frac{1}{4}d b_k\partial_k D_m^2\phi \lambda_m   +\frac{1}{4}\partial_{ij}\phi b_k\partial_k\L_{ij}. 
          	\end{align*}
          	Again, by the eigen-decomposition $\L=\sum_{i=1}^d \lambda_i v_i v_i^T$, we find
          	\begin{equation}
          		\cL\phi(x_0)=b_i\partial_i \phi+\frac{1}{2}\L_{ij}\partial_{ij}\phi |_{x=x_0}=A.
          	\end{equation}
          	Similarly, we find
          	\begin{align*}
          		\frac{1}{2}\cL^2\phi(x_0)& =\frac{1}{2}\big(b_k\partial_kb_i \partial_i\phi+b_kb_i\partial_{ik}\phi
          		+\frac{1}{2}b_k\partial_k\L_{ij}\partial_{ij}\phi+\frac{1}{2}b_k\L_{ij}\partial_{ijk}\phi \big)\\ & \qquad
          		+\frac{1}{4}\L_{kl}\big(\partial_k\partial_l b_i \partial_i\phi
          		+2\partial_kb_i\partial_{il}\phi+b_i\partial_{ikl}\phi
          		+\frac{1}{2}\partial_k\partial_l\L_{ij}\partial_{ij}\phi
          	\\ & \qquad 	+\partial_k\L_{ij}\partial_{ijl}\phi
          		+\frac{1}{2}\L_{ij}\partial_{ijkl}\phi \big),
          	\end{align*}
          	which equals to $B$. Together with (i), Corollary~\ref{cor:second} gives the claim.
          \end{proof}
      
      \smallskip
      
          \begin{remark}\label{rmk:numberofrv}
          	For the multi-dimensional Algorithm~\ref{alg:multdode}, though we have exponentially many Gaussians, the complexity is just linear in $d$. In fact, one needs the number of random variables to grow at least linearly in $d$ to get a weak second-order scheme for general SDEs \cite{mil1979method,talay1984,milstein86}. 
          \end{remark}

          \subsection{Efficiency of the Monte Carlo method.}
          
          For the multi-dimensional Algorithm~\ref{alg:multdode}, though we have exponentially many Gaussians, we see that the complexity is just linear in $d$, which means our algorithm has good computational efficiency.
          Since we only care about the distributions, we often use Monte Carlo methods
          \cite{metropolis1949,rubinstein2016} to generate a large number of samples and use the empirical measure to approximate the probability measure.  As we know, the error and efficiency of Monte Carlo methods depend on the variance. The variance of the Euler-Maruyama scheme \eqref{eq:em} is $\L(x_n)h$, where $x_n$ is the value of the scheme at $t_n$.
          For the same reason, if we can show that the variance of Algorithm~\ref{alg:multdode} after one step is proportional to $h$, then the Monte Carlo method based on our algorithm is as efficient as the Monte Carlo method based on the Euler-Maruyama method \eqref{eq:em}.
          
          In this section, we compute the second moment
          \begin{equation}
          	M_2:=\mathbb{E}\left(|X^{n+1}-x_n|^2 \mid X^n =x_n\right)
          \end{equation}
          and show that it is indeed $O(h)$ despite we have exponentially many Gaussians. For the notational convenience, we define the matrix norm 
\begin{equation}
          		||\L||_{\tr}=\sup_{x\in\mathbb{R}^d} \tr(|\L(x)|),
\end{equation} 
where $|\L(x)|=\sum_{i=1}^d |\lambda_i(x)| v_i(x)v_i^T(x)$ if $\L(x)$ is given by \eqref{eq:Lambdadecomp}. 

\smallskip
          
\begin{proposition}
There exists $h_0>0$ such that when $h<h_0$
\[
M_2\le 3\|\L\|_{\tr}h,
\] for Algorithm~\ref{alg:multdode}. 
 \end{proposition}

\smallskip
          
          \begin{proof} 
          	By \eqref{eq:multdnormalass}, direct computation shows that 
          	\begin{align} \label{eq:M2calculation}
          		M_2 &= \int_{\mathbb{R}} ||x-x_n||^2\bar{\rho}(x)dx 
          		=\sum_{p\in P} w_p \big(\tr(S_p(h)) +\|m_p(h)-x_n\|^2  \big) \\
          		& \le  \sum_{p\in P} w_p\tr(S_p(h))+2h^2\|b\|_{\infty}^2+ 2\sum_{p\in P} w_p \big\|\sum_{i=1}^dz_p^i \sqrt{\frac{3}{2} \lambda_i h}v_i\big\|^2\\
          		& =  \sum_{p\in P} w_p\tr(S_p(h))+2h^2\|b\|_{\infty}^2+\tr(\L(x_n))h.
          	\end{align}
          	The first inequality in \eqref{eq:M2calculation} follows from \eqref{eqn:highdm0} and \eqref{eqn:mphdiffoh}:
          	\begin{align*}
          		|m_p(h)-x_n|^2=|m_p(h)-m_p(0)+m_p(0)-x_n|^2
          		& \le 2(|m_p(h)-m_p(0)|^2+|m_p(0)-x_n|^2) \\ &\le 2h^2\|b\|_\infty^2 + 2\Big\|\sum_{i=1}^d z_p^i \sqrt{\frac{3}{2}\lambda_i h}v_i\Big\|^2.       	\end{align*}
          	For the last equality, we have by the fact that $\{v_i\}$'s are orthonormal:
          	\[
          	\sum_{p\in P} w_p \Big\|\sum_{i=1}^dz_p^i \sqrt{\frac{3}{2} \lambda_i h}v_i\Big\|^2
          	=\frac{3h}{2}\sum_{p\in P} \sum_{i=1}^dw_p |z_p^i|^2 \lambda_i,
          	\]
          	and the last equality in \eqref{eq:M2calculation} follows since $\sum_{p\in P} w_p |z_p^i|^2=2w_1=\frac{1}{3}$ (see \eqref{eq:auxiliaryidentityapp}). Now, noticing $\tr(G(m(t))) \le \frac{3}{2} \|\L\|_{\tr}$, we obtain
          	\[
          	\tr(S_p(h))\le \frac{3}{2}h\|\L\|_{\tr},~p\in P.
          	\]
          	For Algorithm~\ref{alg:multdode}, when $h$ is small enough, we have
          	\[
          	\tr(S_p(h))\le h \|\L\|_{\tr} .
          	\]
          	Since $\|b\|_{\infty}^2h^2$ is in higher order, the claim follows. 
          \end{proof}

          \section{Numerical experiments}\label{sec:numeric}
          
          In this section, we apply the algorithm on SDE \eqref{eq:sde} in It\^o sense with different choices of $b$ and $\sigma$. Note that the Assumption~\ref{ass:bounded} $\sigma, b\in C_b^m$ is only listed for convenience of theoretical analysis. For a diffusion process starting at $x_0$, within finite time $T$, the probability density is concentrated in a finite domain and the far away behaviors of $b$ and $\sigma$ are not important. Hence, in the simulation here, we may use unbounded $b$ and $\sigma$.  We also check how the algorithm behaves 
          if there are some degenerate points of $\L$ (i.e. $\L$ is only positive semi-definite at these points).

          \subsection{A 1D example with regular $\sigma$.}\label{subsec:example1}
          
          This example is designed to test the correctness of Algorithm~\ref{alg:1dode}. 
          The dimension is $d=1$ and $\sigma^2$ is uniformly bounded from below. We will also plot the empirical distribution generated by our algorithm to compare with the one generated by Euler-Maruyama scheme \eqref{eq:em}. 
          
          The SDE we consider is as following:
          \begin{equation}\label{eq:ex1}
          	dX(t)=\lambda X(t)dt +\sqrt{X(t)^2+4}\, dW .
          \end{equation}
          The diffusion coefficient $\sigma(x)=\sqrt{x^2+4}$ is bounded below uniformly so that there is no degenerate point.
          
          To test the correctness of our algorithm, we use the test function $\phi(x)=x^2$ and define the relative error as
          \begin{equation}\label{eq:relativeE}
          	E=\frac{1}{\mathbb{E}X^2(T)}\Big|\frac{1}{N}\sum_{k=1}^N (X^{(k), [T/h]})^2-\mathbb{E}X^2(T) \Big|,
          \end{equation}
          where $X^{(k)}=\{X^{(k), n}\}_{n\ge 0}$ is the sequence generated by the numerical algorithm in the $k$-th experiment. Hence, $X^{(k)}$ is a sample path. The exact expectation $\mathbb{E}_{x_0}X^2(T)$, by It\^o's formula  \cite[Chap. 4]{oksendal03}, is given by 
          \[
          \mathbb{E}_{x_0}X^2(T)=x_0^2 \exp((2\lambda+1)T)+4\frac{\exp((2\lambda+1)T)-1}{2\lambda+1}.
          \]

          In Figure~\ref{fig:1dex1}, we plot the results of the simulation for $X(0)=2, \lambda=-2$ and $T=2$. Each error is computed using $N=10^8$ trajectories. The ``error bars" are obtained by chopping all samples into $10$ slices, with each slice containing $10^7$ trajectories. We then compute the relative error \eqref{eq:relativeE} in each slice, denoted by $E^{(m)}$ ($1\le m\le 10$). We find the standard deviation $\sigma_E$ for the data $\{E^{(m)}\}_{m=1}^{10}$, and use $[E-1.65\sigma_E, E+1.65\sigma_E]$ as our confidence interval. We find that 
          our Gaussian mixture method gives weak second-order accuracy.
          \begin{figure}[ht]
          	\begin{center}
          		\includegraphics[width=0.5\textwidth]{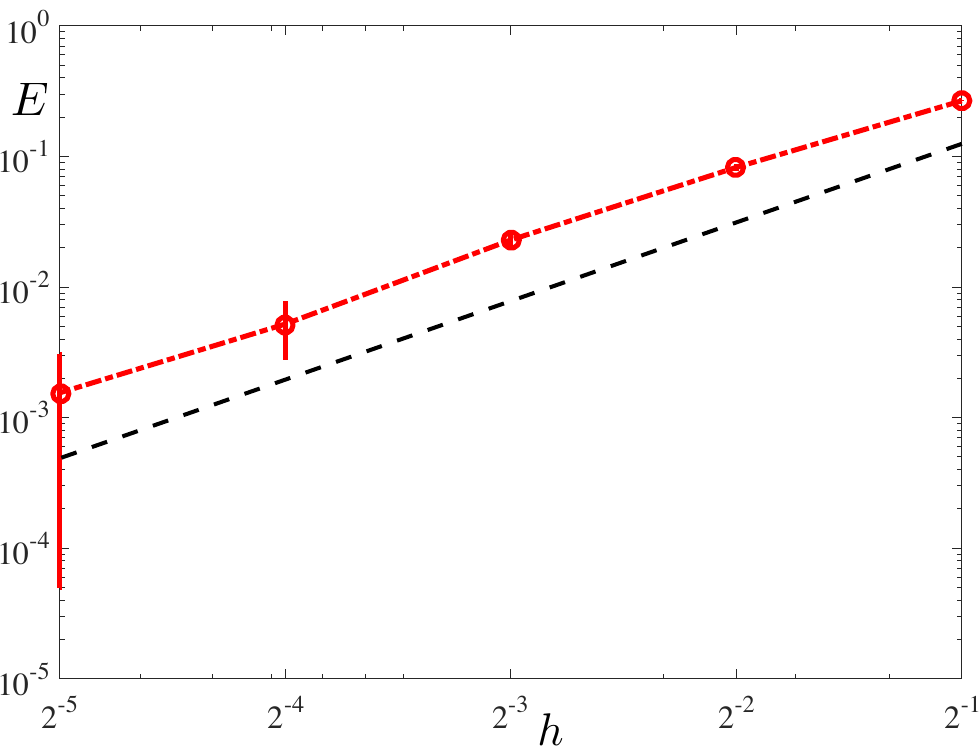}
          	\end{center}
          	\caption{$X(0)=2$, $\lambda=-2$, $T=2$. We plot the errors obtained by the Gaussian mixture method
          		The vertical short segments are the ``error bars'' and the black dashed line indicates $E=h^2$.}
          	\label{fig:1dex1}
          \end{figure}

          To confirm that the Gaussian mixture method gives the desired distribution, we now plot the empirical distribution in Figure~\ref{fig:hist} by histcounts.  All the empirical densities are obtained by using $N=10^6$ points, and the initial condition $X(0)=2$. We take the results obtained from Euler-Maruyama (E-M) scheme \eqref{eq:em} with $\Delta t=h^3$ as the reference density (green curves in Figure~\ref{fig:hist} (a) and (b)). 
          
          In Figure \ref{fig:hist} (a), we plot the empirical densities obtained by Algorithm~\ref{alg:1dode} (red) 
          and Euler-Maruyama (black) after one step with step size $\Delta t=h=1/32$. At time $t=h$, the reference density (green curve) has a peak at $x_c\approx 1.79$ while its mean is located at the black dot ($\bar{x}\approx 1.88$). We also calculated the empirical skewness $\gamma_1 = \mathbb{E}\Big(\dfrac{X(h)-\bar{x}}{\sigma}\Big)^3\approx 0.3695$ (here only $\sigma$ denotes the variance of the reference density), and the kurtosis $K=\mathbb{E}\Big(\dfrac{X(h)-\bar{x}}{\sigma}\Big)^4 \approx 3.3078 $, while the accurate skewness and kurtosis are 0.3718 and 3.3153 respectively. The skewness and kurtosis for a Gaussian (Euler-Maruyama method) are 0 and 3 respectively. For Algorithm~\ref{alg:1dode}, these two numbers are 0.3717 and 3.1888.
          
          In Figure \ref{fig:hist} (b), we plot the empirical densities obtained by Algorithm~\ref{alg:1dode} (red) 
          and Euler-Maruyama (black) at time $t=1$ with step size $\Delta t=h=1/32$. We find that the densities given by our weak second-order algorithm almost coincides with the reference density, while the one given by E-M is worse.
          \begin{figure}[ht]
          	\begin{center}
          		\includegraphics[width=0.5\textwidth]{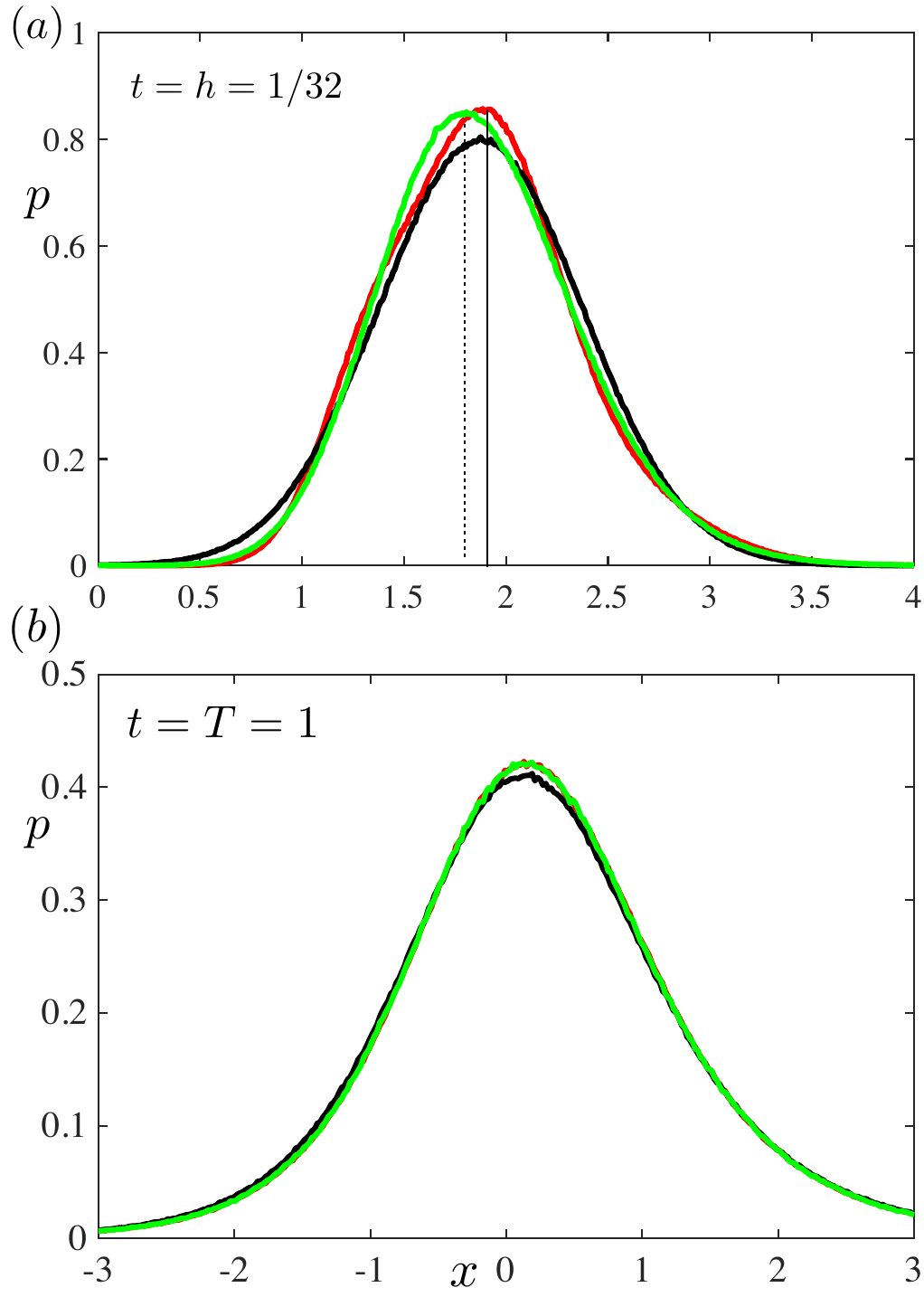}
          	\end{center}
          	\caption{The green lines are the `accurate densities' obtained using E-M with step size $\Delta t=h^3$. Other curves are empirical distributions (black dashed line: E-M;  red broken line: Gaussian mixture) 
          		obtained with step size $\Delta t=h$. (a). Empirical distributions after one step. The solid vertical line shows the mean of the green curve while the dashed line shows the peak. Empirical skewness is 0.3695 and kurtosis is 3.3078. (b). Empirical distributions are at $t=1$.}
          	\label{fig:hist}
          \end{figure}
          
          To sum up, for this example \eqref{eq:ex1}, the Gaussian mixture method has weak second-order accuracy and is able to capture the correct distribution better. 

\subsection{1D Geometric Brownian Motion.}
In this example, we consider the 1D Geometric Brownian Motion 
\[
dX(t)=\lambda X(t)\,dt+\sigma X(t)\,dW,
\]
which has a degenerate diffusion coefficient 
\[
\sigma(x)=\sigma x.
\]
          
          Again, we test the weak accuracy with test function $\phi(x)=x^2$ and define the weak error
          \[
          E=\frac{1}{\mathbb{E}X^2(T)}\Big|\frac{1}{N}\sum_{k=1}^N (X^{(k), [T/h]})^2-\mathbb{E}X^2(T)\Big|.
          \]
          By It\^o calculus, it is straightforward to find
          \[
          \mathbb{E}X^2(T)=x_0^2\exp((2\lambda+\sigma^2)T).
          \]
          In Figure~\ref{fig:gbm1d}, we plot the weak error of simulations for $\lambda=-0.8, \sigma=0.85, x_0=5, T=1$ with $N=2\times 10^8$.  The error bars are computed by slicing the samples into $5$ pieces of equal size, and the method is the same as in Section~\ref{subsec:example1} (confidence interval is $[E-1.65\sigma_E, E+1.65\sigma_E]$).  
          
          For the tested parameters our Gaussian mixture method 
          still works and is of weak second-order.  For this example, 
          the error of Algorithm~\ref{alg:1dode} scales like $h^2$ only when $h$ becomes small.
          This can be seen in the kink in Figure~\ref{fig:gbm1d} where only the left-most
          two points seem to line up with the order $h^2$ line. After further
          investigation, we find that for the first three $h$ values ($h=0.25,
          0.125, 0.0833$), there are roughly $1/6$ chance that the computed
          $S(h)$ from the ODE is negative. For smaller values of $h$ ($h=0.0625,
          0.05$), $S(h)$ is always nonnegative for the samples we have.
          In light of this, we
          expect the second-order behavior for our approach to appear in the
          examples with $h\lesssim 0.0625$. When there is a resonable chance
          that  $\sigma^2$ is degenerate, our approach seems to lose the second-order accuracy.

          \begin{figure}[htbp]
          	\begin{center}
          		\includegraphics[width=0.6\textwidth]{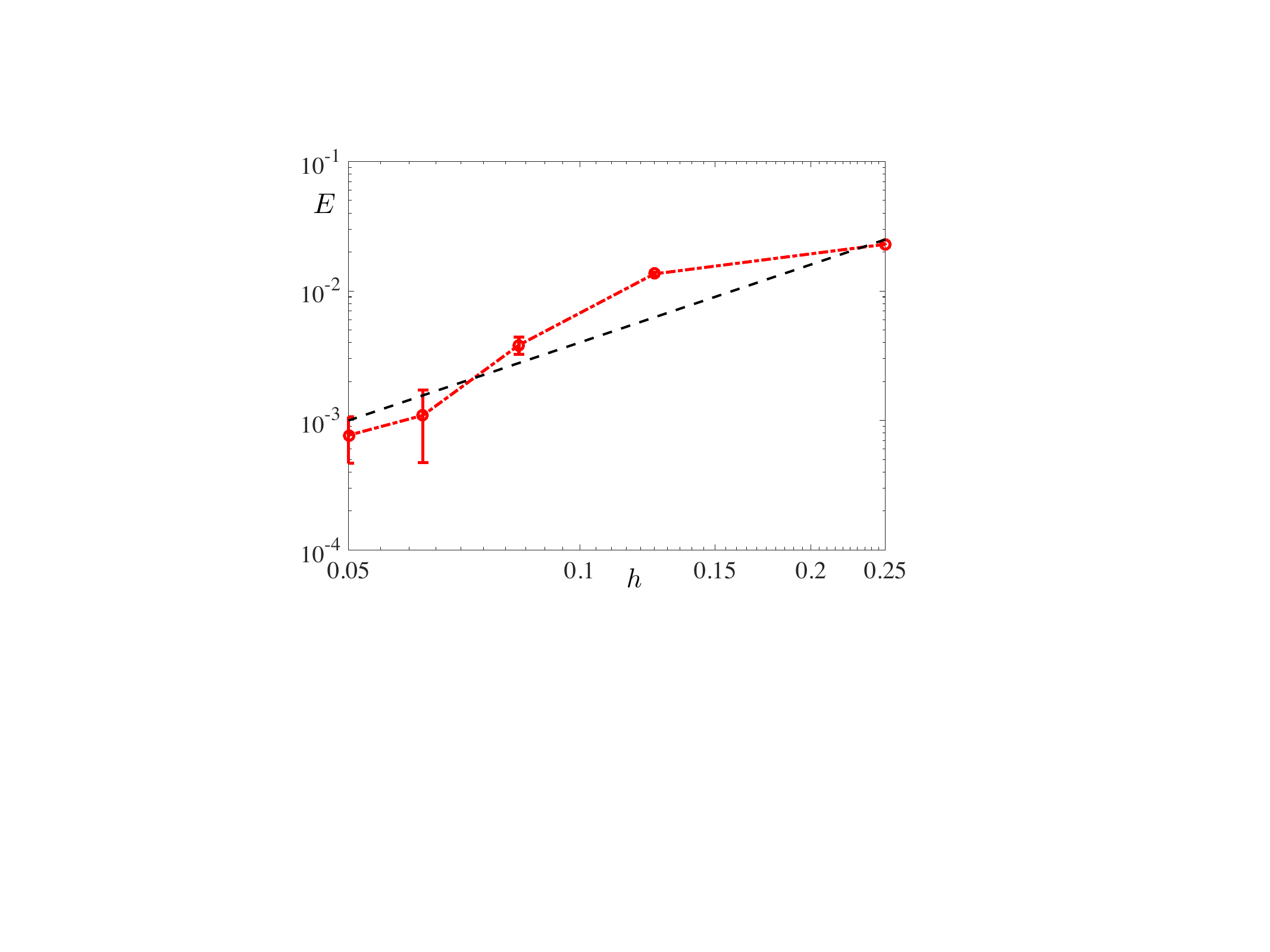}
          	\end{center}
          	\caption{Log-Log error plot for geometric Brownian motion, with $X(0)=5, \lambda=-0.8, \sigma=0.85 , T=1$. The red line with circles is the error obtained by the Gauss mixture method. 
          		The black dashed line is $E=0.4h^2$ and vertical bars represent the ``error bars''.}
          	\label{fig:gbm1d}
          \end{figure}

          \subsection{A 2D example.}
          
          In this example, we consider a 2D SDE, which is a modification of the first example in \cite{am11}:
          \begin{equation}\label{eq:2d1}
          	\begin{pmatrix}
          		dX_1(t) \\ dX_2(t)
          	\end{pmatrix}=\begin{pmatrix}
          		X_1(t) \\ -X_2(t)
          	\end{pmatrix} + X_1(t) \begin{pmatrix}
          		0 \\ 1
          	\end{pmatrix}dW_1(t) 
          	+\sigma
          	\begin{pmatrix}
          		1 \\ 1
          	\end{pmatrix}
          	dW_2(t) ,
          \end{equation}
          where $W_1(t)$ and $W_2(t)$ are independent standard Brownian motions, and $\sigma$ is a positive constant. The purpose here is to show that our Gaussian mixture method for multi-dimensions (Algorithm~\ref{alg:multdode}) 
          works for $\L(x)$ that has varying eigen-directions. 
          
          We consider the solution of \eqref{eq:2d1} at $T=1$ with initial condition $X_1(0)=X_2(0)=1$ and $\sigma=0.1$. We will use the test function $\phi(x)=x_2^2$ to check the weak accuracy.  By It\^o's formula, 
          \[
          \mathbb{E}X_2^2(t)=e^{-2t}\big(\mathbb{E}X_2^2(0)-\frac{1}{4}\mathbb{E}X_1^2(0)-\frac{3\sigma^2}{8}\big)
          +\frac{\sigma^2}{4}+e^{2t}\big(\frac{1}{4}\mathbb{E}X_1^2(0)+\frac{\sigma^2}{8}\big).
          \]
          As before, 
          the relative error is computed as 
          \[
          E=\frac{1}{\mathbb{E}X_2^2(T)}\Big|\frac{1}{N}\sum_{k=1}^N (X_2^{(k), [T/h]})^2-\mathbb{E}X_2^2(T) \Big|.
          \] 
          
          In Figure~\ref{fig:2D1}, we sketch the error plots with $N=2\times 10^8$ and also slice these samples into $10$ equal pieces for the ``error bar'' calculation (confidence interval $[E-1.65\sigma_E, E+1.65\sigma_E]$ and $\sigma_E$ is the standard deviation for these 10 data). We find that our Gaussian mixture method gives weak second-order accuracy for this 2D example as well. 
          \begin{figure}[htbp]
          	\begin{center}
          		\includegraphics[width=0.6\textwidth]{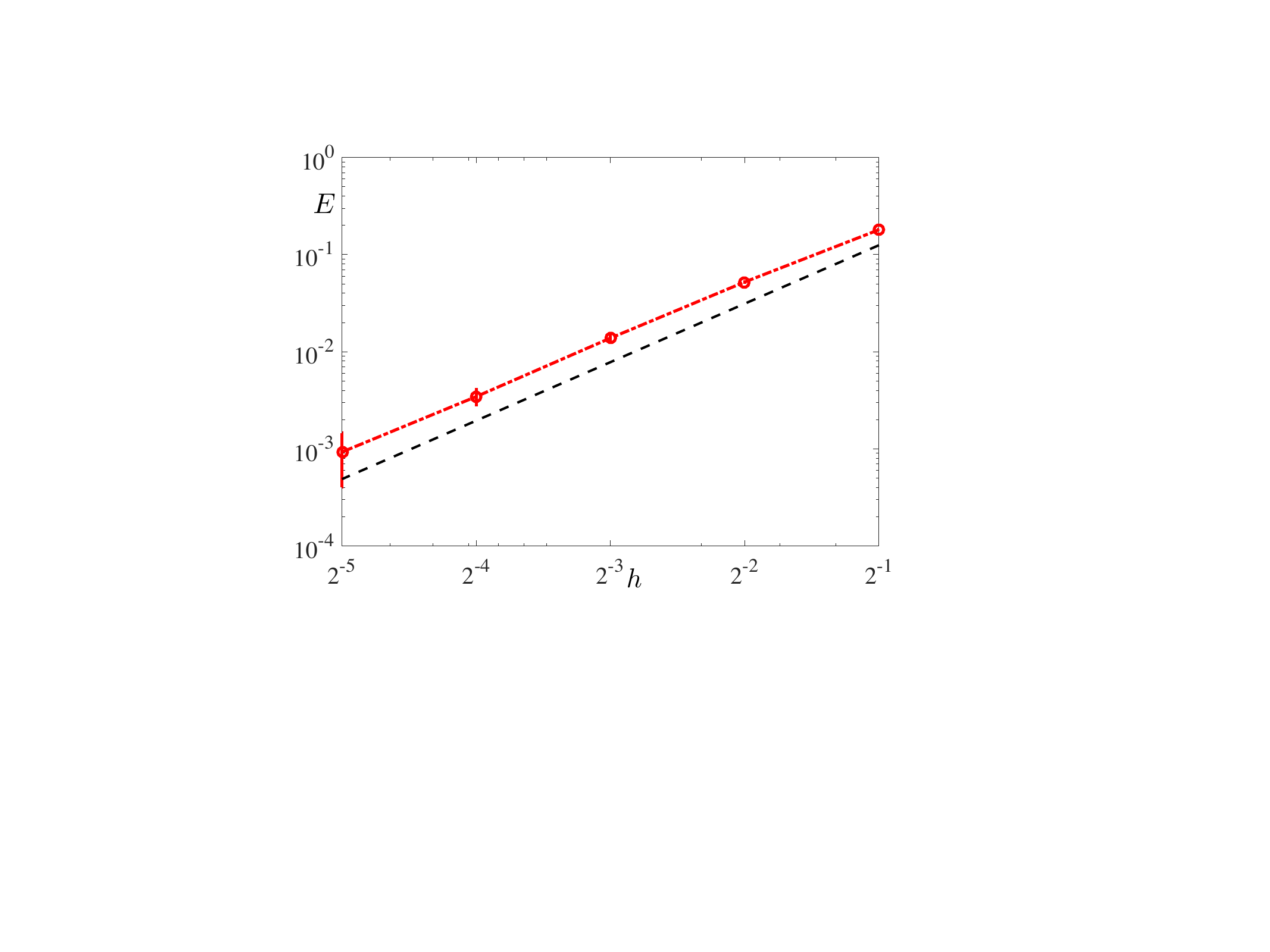}
          	\end{center}
          	\caption{Log-log error plot of the Gaussian mixture method for the 2D example with $\sigma=0.1$ (red line). 
          		The black line is $E=0.5h^2$ while the vertical segments are the ``error bars''. }
          	\label{fig:2D1}
          \end{figure}

          
          
          \subsection{A 6D Example.}

          According to Algorithm~\ref{alg:multdode}
          , the proposed Gaussian mixture method depends explicitly on the dimension and one is surely curious with what will happen if the dimension gets higher. In this example, we look at a 6D problem and verify that our algorithm is still weak second-order.
          
          The SDE we consider is given by:
          \begin{multline}
          	d\begin{pmatrix}
          		X_1 \\ X_2 \\ X_3 \\ X_4 \\ X_5 \\ X_6
          	\end{pmatrix}=\begin{pmatrix}
          		-1 & 1 & 0 & 0 & 0 & -1 \\ -1 & -1 & 1 & 0 & 0 & 0 \\ 0 & -1 & -1 & 1 & 0 & 0 \\ 0 & 0 & -1 & -1 & 1 & 0 \\ 0 & 0 & 0 & -1 & -1 & 1 \\ 1 & 0 & 0 & 0 & -1 & -1
          	\end{pmatrix}\begin{pmatrix}
          		X_1 \\ X_2 \\ X_3 \\ X_4 \\ X_5 \\ X_6
          	\end{pmatrix}dt +\\ 
          	\sigma \begin{pmatrix}
          		\sqrt{0.1+X_1^2} & -0.1 & 0 & 0 & 0 & -0.1 \\ -0.1 & \sqrt{0.2+X_2^2} & -0.1 & 0 & 0 & 0 \\ 0 & -0.1 & \sqrt{0.3+X_3^2} & -0.1 & 0 & 0 \\ 0 & 0 & -0.1 & \sqrt{0.4+X_4^2} & -0.1 & 0 \\ 0 & 0 & 0 & -0.1 & \sqrt{0.5+X_5^2} & -0.1 \\ -0.1 & 0 & 0 & 0 & -0.1 & \sqrt{0.6+X_6^2}
          	\end{pmatrix}d\begin{pmatrix}
          		W_1 \\ W_2 \\ W_3 \\ W_4 \\ W_5 \\ W_6
          	\end{pmatrix}
          \end{multline}

          We take $\sigma=0.7$ and check the solution at $t=2$. The initial condition we use is $X_i(0)=1$ for all $1\le i \le 6$.
          The test function we use is 
          \begin{equation}
          	\phi(x)=\sum_{i=1}^6 x_i^2.
          \end{equation} 
          By It\^{o}'s formula
          \begin{equation}
          	\mathbb{E}\phi(X(T))=(\sum_{i=1}^6 X_i^2(0))\exp((-2+\sigma^2)t)+\dfrac{2.22\sigma^2}{\sigma^2-2}(\exp((-2+\sigma^2)t)-1).
          \end{equation}
          The relative error is again defined as
          \begin{equation}
          	E=\frac{1}{\mathbb{E}\phi(X(T))}\Big|\frac{1}{N}\phi(X^{(k), [T/h]})-\mathbb{E}\phi(X(T)) \Big|.
          \end{equation}
          
          For the following log-log error plot (Figure~\ref{fig:6D}), we choose $h=\dfrac{1}{4k}$, $1\le k \le 5$. The sample size is $N=2\times 10^8$ for $h\ge \dfrac{1}{16}$ and $5\times 10^8$ for $h=\dfrac{1}{20}$, chopped into $10$ equal slices to produce the error bars with confidence interval $[E-1.65\sigma_E, E+1.65\sigma_E]$ ($\sigma_E$ is again the standard deviation of these $10$ data). The plot demonstrates that the scheme works in high dimensions as well.
          \begin{figure}[htbp]
          	\begin{center}
          		\includegraphics[width=0.6\textwidth]{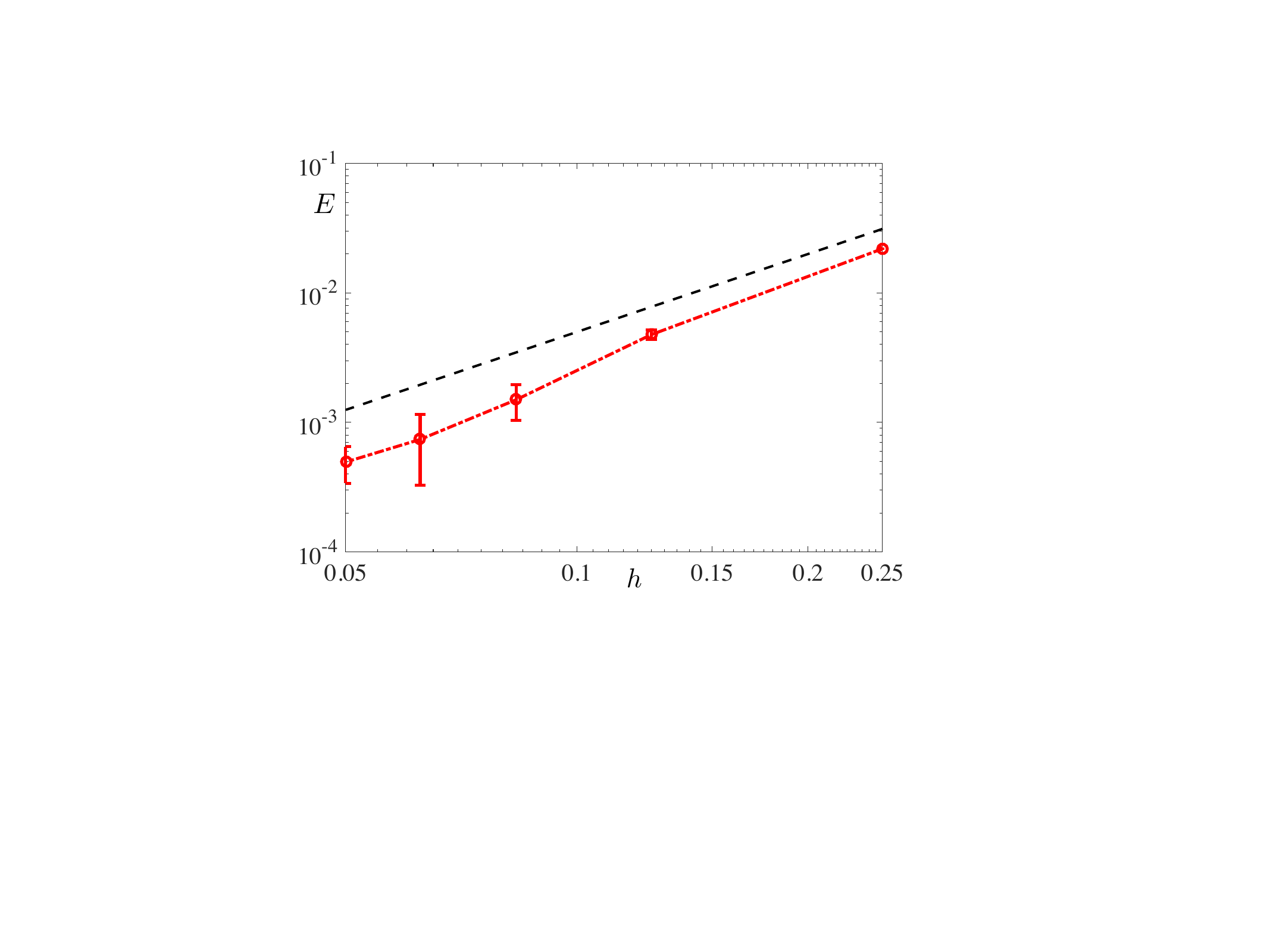}
          	\end{center}
          	\caption{Log-log error plot of Gaussian mixture method (red line). 
          		The black dashed line is $E=0.5h^2$ and the vertical segments indicate the ``error bars''. }
          	\label{fig:6D}
          \end{figure}

          \section*{Acknowledgements}
          
          The work of L. Li is partially sponsored by NSFC 11901389,11971314, and Shanghai Sailing Program 19YF1421300.
          The work of J.~Lu and L.~Wang is supported in part by National Science
          Foundation under grant DMS-1454939, while J. ~Mattingly is supported
          in part by National Science Foundation under grant DMS-1613337.

          \appendix
          
          \section{Proof of Proposition~\ref{pro:globalerror}}\label{app:globalerr}
          \begin{proof}
          	Let us fix $\phi\in C_b^{2(r+1)}$ and define 
          		\begin{align*}
          			& u^n(x)=\mathbb{E}_x\phi(X^n),\\
          			& u(x, t)=\mathbb{E}_x\phi(X(t)).
          		\end{align*}
          	By the Markov property, we have
          	\begin{equation}\label{eq:unsg}
          		u^{n+1}(x)=\mathbb{E}_x(\mathbb{E}_{x}(\phi(X^{n+1}) | X^1))=\mathbb{E}_x(u^n(X^1)).
          	\end{equation}
          	Similarly, we have
          	\begin{equation}\label{eq:utsg}
          		u(x, (n+1)h)=\mathbb{E}_x(u(X(t_1), nh)).
          	\end{equation}
          	Note that $u$ satisfies the backward Kolmogorov equation
          	\[
          	u_t=\cL u=b\cdot\nabla u+\frac{1}{2}\L_{ij}\partial_{ij}u,
          	\]
          	with initial condition
          	\[
          	u(x, 0)=\phi(x).
          	\]
          	By standard parabolic PDE theory, for $b, \sigma\in C_b^{2(r+1)}$, we have
          	\begin{equation}\label{eq:uniformbound}
          		\sup_{0\le t\le T}\|u\|_{C^{2(r+1)}}\le C(T).
          	\end{equation}
          	By Equations \eqref{eq:unsg} and \eqref{eq:utsg}, we have for all $x\in\mathbb{R}^d$ that
          	\begin{align*}
          	|u^{n+1}(x)-u(x, (n+1)h)| & \le |\mathbb{E}_x(u^n(X^1)-u(X^1, nh))| \\ & \qquad +|\mathbb{E}_x(u(X^1, nh))-\mathbb{E}_x(u(X(t_1), nh))|.
          	\end{align*}
          	Define
          	\[
          	E^n=\sup_{x\in \mathbb{R}^d}|u^n(x)-u(x, nh)|,
          	\]
          	by the assumption of Proposition~\ref{pro:globalerror} on local truncation error and Equation \eqref{eq:uniformbound} we have 
          	\begin{equation*}
          		E^{n+1}\le \mathbb{E}_xE^n+|\mathbb{E}_x(u(X^1, nh))-\mathbb{E}_x(u(X(t_1), nh))|
          		\le E^n+C h^{r+1},
          	\end{equation*}
          	where $C=\sup_{0\le t\le T}\rho(\|u(\cdot, t)\|_{C^{2(r+1)}})$. 	This further implies that
          	\[
          	\sup_{n: nh\le T}E^n \le C_1 h^r.
          	\]
          \end{proof}
          
          \section{Proof of Proposition \ref{pro:firstorderofonegauss}}\label{app:firstgaussian}
          
          \begin{proof}
          	For the convenience of notations, we will drop the dependence on $x_0$ so that $m(h)$ indeed means $m(h,x_0)$ and $m_1$ means $m_1(x_0)$ and so on. Denote $\cL_1:=(m(h)-x_0)\partial_x+\frac{1}{2}S(h)\partial_{xx}$, and we have
          	\[
          	\mathbb{E}_{x_0}(\phi(X^1))=\phi(x_0)+\cL_1\phi(x_0)+\frac{1}{2}\cL_1^2\phi(x_0)+O(h^3).
          	\]
          	It follows that
          	\begin{equation}\label{eq:numer1}
          		\mathbb{E}_{x_0}(\phi(X^1))=:\phi(x_0)+Bh+Ch^2+O(h^3).
          	\end{equation}
          	Here, $B$ and $C$ are the coefficients of $h$ and $h^2$:
          		\begin{align*}
          			& B=\phi'(x_0)m_1+\frac{1}{2}\phi''(x_0)S_0,\\
          			& C=\frac{1}{2}\big(\phi''(x_0)m_1^2+\phi'(x_0)m_2\big) 
          			+\frac{1}{2}\big(\frac{1}{2}\phi''(x_0)S_1+\phi'''(x_0)m_1S_0\big)+\frac{1}{8}\phi^{(4)}(x_0)S_0^2.
          		\end{align*}
          	To satisfy the condition \eqref{eq:basic2ndorder}, we need to have
          	\begin{equation}
          		B=\cL\phi(x_0), ~~C=\frac{1}{2}\cL^2\phi(x_0).
          	\end{equation}
          	Recall that $\cL=b\,\partial_x+\frac{1}{2}\L(x)\partial_x^2$, so
          	$B=\cL\phi(x_0)$ requires that for any sufficiently smooth $\phi$,
          	\begin{equation*}
          		b(x_0)\phi'(x_0)+\frac{1}{2}\L(x_0)\phi''(x_0)=\phi'(x_0)m_1+\frac{1}{2}\phi''(x_0)S_0,
          	\end{equation*}
          	which requires
          	\begin{equation}
          		m_1=b(x_0),~S_0=\L(x_0).
          	\end{equation}
          	On the other hand, the requirement $C=\frac{1}{2}\cL^2\phi(x_0)$ can be expanded as
          	\begin{multline*}
          		\frac{1}{2}\Big(b(x_0)(b\phi'+\frac{1}{2}\L \phi'')'+\frac{1}{2}\L(x_0)(b\phi'+\frac{1}{2}\L \phi'')''\Big)\Big|_{x=x_0} \\
          		=\Big(\frac{1}{2}(\phi''m_1^2+\phi'm_2)
          		+\frac{1}{2}(\frac{1}{2}\phi''S_1+\phi'''m_1S_0)+\frac{1}{8}\phi^{(4)}S_0^2 \Big)\Big|_{x=x_0}.
          	\end{multline*}
          	This is impossible in general. For example, the coefficient of $\phi'''$ on right hand side is  $\frac{1}{2}m_1S_0$, or $\frac{1}{2}b(x_0)\L(x_0)$ but the one on left hand side is $\frac{1}{2}b(x_0)\L(x_0)+\frac{1}{4}\L(x_0)\L'(x_0)$. They can not balance unless the diffusion matrix $\L(x)$ is constant.
          	
          \end{proof}
          
          \section{Proof of Lemma \ref{lmm:usefullmm}}\label{app:proofusefullmm}
          
          \begin{proof}         	
          	In this proof, we will again use $R$ to denote a generic function that can depend on the $C^6$ norm of the test function but can be bounded uniformly in $x_0$ and $h$. However, its concrete meaning can change from line to line.
          	
          	Clearly, due to the symmetry, we only need to prove that for all $\phi\in C_b^{\infty}$,
\begin{align}
          			\sum_{p\in P}w_p \phi(x_p)
          			& =\phi(x_0)+\sum_{i=1}^d w^1 D_{i}^2\phi(x_0) \gamma \lambda_i h
          			+\sum_{i<j} (w^1)^2 D_i^2D_j^2\phi(x_0) \gamma^2 \lambda_i\lambda_j h^2  \nonumber\\ & \qquad 
          			+\frac{1}{12}\sum_{i=1}^d w^1 D_{i}^4\phi(x_0) \gamma^2\lambda_i^2h^2
          			+R(x_0, h)h^3.\label{eq:target}
          	\end{align}
          	Without loss of generality, we set $x_0=0.$ With Equation \eqref{eq:weightsmd}, it is convenient to denote the left hand side of \eqref{eq:target} as
          	\begin{equation}
          		T_d(\phi)=\sum_{p\in P} (\prod_{i=1}^d w^{z_p^i})\phi\left(\sum_{i=1}^d z_p^i \sqrt{\gamma \lambda_i h} v_i\right).
          	\end{equation}
          	If $d=1$, the claim follows from 1D Taylor expansion derived in Section \ref{sec:gauss}.
          	Assume that the claim is valid for all $d=1,2,\ldots, m$, $m\ge 1$, and we want to prove for $d=m+1$. Define $P_m$ to be the index set with $d=m$. We find by definition:
          	\begin{align*}
          		T_{m+1}(\phi)&=w^1\sum_{p\in P_m}(\prod_{i=1}^m w^{z_p^i})\phi\left(\sum_{i=1}^m z_p^i \sqrt{\gamma \lambda_i h} v_i+\sqrt{\gamma \lambda_{m+1} h}\right) \\ & \qquad +
          		w^0\sum_{p\in P_m}(\prod_{i=1}^m w^{z_p^i})\phi(\sum_{i=1}^m z_p^i \sqrt{\gamma\lambda_i h} v_i)\\ & \qquad 
          		+w^{-1}\sum_{p\in P_m}(\prod_{i=1}^m w^{z_p^i})\phi(\sum_{i=1}^m z_p^i \sqrt{\gamma\lambda_i h} v_i-\sqrt{\gamma\lambda_{m+1} h}).
          	\end{align*}
          	For each $p$, we do Taylor expansion of $\phi$ about $\sum_{i=1}^m z_p^i \sqrt{\gamma \lambda_i h} v_i$ and have
          	\begin{equation}
          		T_{m+1}(\phi)=T_m(\phi)+w^1T_m(D_{m+1}^2\phi)\gamma \lambda_{m+1}h
          		+\frac{1}{12}w^1T_m(D_{m+1}^4\phi)\gamma^2\lambda_{m+1}^2h^2+Rh^3.
          	\end{equation}
          	By the induction hypothesis, we have
          	\begin{align*}
          		T_{m+1}(\phi)& =\phi(0)+\sum_{i=1}^m w^1 D_{i}^2\phi \gamma \lambda_i h 
          		+\sum_{i<j} (w^1)^2 D_i^2D_j^2 \gamma^2 \lambda_i\lambda_j h^2 
          		+\frac{1}{12}\sum_{i=1}^m w^1 D_{i}^4\phi \gamma^2\lambda_i^2h^2 \\ & \qquad 
          		+w^1\bigl(D_{m+1}^2\phi+\sum_{i=1}^m w^1D_i^2D_{m+1}^2\phi  \gamma\lambda_i h \bigr)\gamma \lambda_{m+1}h 
          		\\ & \qquad +\frac{1}{12}w^1D_{m+1}^4\phi\gamma^2\lambda_{m+1}^2h^2+R h^3.
          	\end{align*}
          	Arranging the terms on the right hand side, we find the claim is also true for $d=m+1$.
          \end{proof}

          \section{A variance construction approach}\label{app:varconstruct}
          
          \subsection{The variance construction method for one dimension.}
          Motivated by \eqref{eq:Sij} and \eqref{eq:theconstraints},
          we can construct
          	\begin{align*}
          		& S_{i}(h)=\frac{h}{2}\L(x_0+z_i\sqrt{6\L(x_0)h}+hb(x_0)),
          		~i=\pm 1, \\
          		& S_0(h)=\frac{1}{2}h\L(x_0+b(x_0)h)-\frac{3}{8}\L(x_0)\L''(x_0)h^2
          		+\frac{(3\L(x_0)\L''(x_0)/8)^2}{\L(x_0+b(x_0)h)/2} h^3.
          	\end{align*}
          We can verify that the constraints are all satisfied. The third term added is to ensure that $S_0$ is non-negative.
          Compared with the ODE flow method, the drawback of this method is that it involves higher order spatial derivatives, such as  $\L''$. In practice, one may approximate it by finite difference $\frac{1}{h^2} \bigl(\L(x_0+h)-2\L(x_0)+\L(x_0-h)\bigr)$.
          
          \begin{remark}
          	The third correction term can be thrown away if $h$ is small enough. For example, 
          	\[
          	h< \frac{4\inf_x|\L(x)|}{3\|\L \L'' \|_{\infty}} .
          	\]
          \end{remark}
          
          This construction gives the following Algorithm \ref{alg:1ddirectS} to generate $x_{n+1}$ given $X^n=x_n$.
          \begin{algorithm}
          	\caption{Gaussian mixture scheme for SDEs (variance construction method in 1D)\label{alg:1ddirectS}}
          	\begin{algorithmic}[1]
          		\State Generate $z$ such that $P(z=0)=\frac{2}{3}$
          		and $P(z=1)=P(z=-1)=\frac{1}{6}$. Then,
          		\begin{equation*}
          			m(0)=x_n+z\sqrt{\frac{3}{2} \L h}
          		\end{equation*}         		
          		If $z=0$, 
          		\begin{equation*}
          			S(h)=\frac{1}{2}h\L(x_n+b(x_n)h)-\frac{3}{8}\L(x_n)\L''(x_n)h^2
          			+\frac{(3\L(x_n)\L''(x_n)/8)^2}{\L(x_n+b(x_n)h)/2} h^3.
          		\end{equation*}
          		Otherwise,
          		\begin{equation*}
          			S(h)=\frac{h}{2}\L\big(x_n+z\sqrt{6\L(x_n)h}+hb(x_n)\big).
          		\end{equation*}
          		
          		\State Solve the ODE $\dot{m}=b(m)$ with the initial value $m(0)$ using a scheme of at least second order to obtain $m(h)$.
          		
          		\State
          		Sample 
          		\begin{equation}
          			x_{n+1}=m(h)+\sqrt{S(h)}\xi,
          		\end{equation}
          		where $\xi$ is a standard 1D normal variable.
          	\end{algorithmic}
          \end{algorithm}
          One can verify that the requirements in \eqref{eq:2ndCond2} are satisfied, which gives the following theorem:
          
          \smallskip
          
          \begin{theorem}
          	Let $d=1$. Suppose Assumptions~\ref{ass:pd}-\ref{ass:bounded} hold, then Algorithm~\ref{alg:1ddirectS} is a weak second-order scheme for the 1D diffusion process \eqref{eq:sde}.
          \end{theorem}The proof is identical to that of Theorem \ref{thm:1dodeflow} and is omitted here.

\subsection{The variance construction method for multi-dimension.}
          
          As before, one may want to guarantee that $S_p(h)$ is positive definite for $p\in P$.  We now present a variance construction method for $S_p(h)$ for multi-dimension. Consider that $m_p(h)$ and $S_p(h)$ are given by 
          \begin{equation*}
          	\begin{aligned}
          		& \dot{m}_p=b(m_p(t)), ~ m_p(0)=y_p,\\
          		& S_p(h)=\frac{h}{2}\L\Big(x_0+\sum_{i=1}^d z_p^i\sqrt{6\lambda_i h}v_i+hb\Big)-\frac{3h^2}{8}\sum_{i=1}^d(1-|z_p^i|)\lambda_iD_{i}^2\L +F_p(h)h^3,
          	\end{aligned}
          \end{equation*}
          where $\sum_{i=1}^d(1-|z_p^i|)\lambda_iD_{i}^2\L$ can be approximated by finite difference. In particular, if we set $\theta=\sum_{i=1}^d\sqrt{(1-|z_p^i|)\lambda_i}v_i$, then
          \[
          \sum_{i=1}^d(1-|z_p^i|)\lambda_iD_{i}^2\L \approx \frac{1}{h^2}\Big(\L(x_0+h\theta)-2\L(x_0)+\L(x_0-h\theta)\Big).
          \]
          $F_p(h)h^3$ is added to ensure
          that $S_p$ is positive semi-definite. Let the first two terms in $S_p$
          be $hA_p$ and $h^2B_p$, where $A_p$ is positive definite and thus invertible. Then, we have
          \begin{equation}\label{eq:F}
          	F_p=\frac{1}{4}B_pA_p^{-1}B_p.
          \end{equation}
          
          We propose Algorithm~\ref{alg:multddirectS} to generate $x_{n+1}$ given $X^n=x_n$.
          
          \begin{algorithm}[htb]
          	\caption{Gaussian mixture scheme for SDEs (variance construction method in higher D)\label{alg:multddirectS}}
          	\begin{algorithmic}[1]
          		\State Using some fast algorithm, decompose 
          		\[
          		\L(x_n)=\sum_{i=1}^d \lambda_i(x_n) v_i(x_n)v_i^T(x_n).
          		\]
          		
          		\State Generate $z^i, i=1,2,\ldots, d$ so that $P(z^i=0)=\frac{2}{3}$ while $P(z^i=1)=P(z^i=-1)=\frac{1}{6}$.
          		
          		\State Construct the matrix
          		\begin{equation}
          			S(h)=\frac{h}{2}\L(x_n+\sum_{i=1}^d z^i \sqrt{6 \lambda_i h} v_i)-\frac{3h^2}{8}\sum_{i=1}^d (1-|z^i|)\lambda_iD_{i}^2\L+F(h)h^3
          		\end{equation}
          		where $F$ is constructed according to \eqref{eq:F}.
          		
          		\State
          		Let
          		\begin{equation*}
          			m(0)=x_n+\sum_{i=1}^d z^i \sqrt{\frac{3}{2} \lambda_i h} v_i.
          		\end{equation*}
          		and then 
          		$m(h)$ is obtained by solving $\dot{m}=b(m)$ using an ODE solver with at least second-order accuracy (e.g. RK2, RK4).
          		
          		\State Sample $x_{n+1}\sim \mathcal{N}(m(h), S(h))$, or 
          		\[
          		x_{n+1}=m(h)+\sum_{i=1}^d \sqrt{\mu_i^+}\xi_i u_i .
          		\]
          		where $S(h)=\sum_{i=1}^d \mu_i u_i u_i^T$ with $\{u_i\}_{i=1}^d$ being orthonormal, and $\{\xi_i\}$ are i.i.d standard 1D normal variables.
          	\end{algorithmic}
          \end{algorithm}
      
         \smallskip
          
          \begin{remark}
          	Notice that we need to invert a matrix to get $F(h)$, which is not desirable when $d$ is large. However, similar as the 1D case, if $h$ is small enough, $F(h)h^3$ can be thrown away and we can still guarantee the positive definiteness.
          \end{remark}
      
      \medskip
          
          \begin{theorem}
          	Suppose Assumptions~\ref{ass:pd}-\ref{ass:bounded} hold, then Algorithm~\ref{alg:multddirectS} is a second-order scheme for the multi-dimensional diffusion process \eqref{eq:sde}.
          \end{theorem}
      
      \smallskip
          
          \begin{proof}
          	Again, the idea is to check the conditions in Corollary~\ref{cor:second}. Our strategy is not to verify the conditions directly, instead, we compare it to Algorithm~\ref{alg:multdode}, the one using an ODE approach.
          	
          	Again, we only have to check $X^1$ given $X^0=x_0$. Let $S_p^o$ be the covariance matrix obtained following Algorithm~\ref{alg:multdode} at time $h$ while $S_p^s$ be the covariance matrix constructed in this section at time $h$ for $p\in P$. Let $\mathbb{E}_{x_0}^s$ denotes the expectation under the process constructed here while $\mathbb{E}_{x_0}^o$ be the expectation under the process in Algorithm~\ref{alg:multdode}.
          	
          	Consider
          	\[
          	E=\mathbb{E}_{x_0}^s(\phi(X^{1}))
          	-\mathbb{E}_{x_0}^o(\phi(X^{1}))=:E_1-E_2.
          	\]
          	Since the two algorithms only give different covariance matrices, we have by Equation \eqref{eq:multdnormalass}:
          	\begin{multline}\label{eq:errortwoalg}
          		|E|  \le \Big|\frac{1}{2}\sum_{p\in P} w_p \partial_i\partial_j\phi(m_p(h))(S_{p,ij}^o-S_{p,ij}^s) \\
          		+\frac{1}{8}\sum_{p\in P} w_p \partial_{ijkl}\phi(m_p(h))(S_{p,ij}^oS_{p,kl}^o -S_{p,ij}^sS_{p,kl}^s)\Big|+R(h)h^3.
          	\end{multline}
          	We denote 
          	\begin{equation}\label{eq:defeta}
          		\eta_p=(m_p(0)-x_0)/\sqrt{h}=\sum_{i=1}^d z_p^i \sqrt{\frac{3}{2}\lambda_i }v_i.
          	\end{equation}
          	By Equation \eqref{eq:mdodesys} and direction Taylor expansion on $t$, we find for $p\in P$,
          	\begin{multline*}
          		S_p^o=G(m_p(0))h+\frac{1}{2}b(m_p(0))\cdot\nabla G(m_p(0))h^2
          		+R(h)h^3\\
          		=\frac{1}{2}h\L+\eta_p\cdot\nabla G h^{3/2}
          		+\frac{1}{2} (\eta_p)_i(\eta_p)_j\partial_{ij}G h^2
          		+\frac{1}{2}b\cdot\nabla G h^2+K_p^1(h)h^{5/2}+R(h)h^3,
          	\end{multline*}
          	where $K_p^1(h)$ is a bounded function.
          	
          	We do expansion on $S_p^s$ and have
          	\begin{align*}
          		S_p^s & =\frac{1}{2}h\L+\eta_p\cdot\nabla G h^{3/2}
          		+\frac{1}{2}b\cdot\nabla \L h^2
          		+(\eta_p)_i(\eta_p)_j\partial_{ij}\L h^2
          		-\frac{3h^2}{8}\sum_{i=1}^d(1-|z_p^i|)\lambda_i D_i^2\L \nonumber \\ & \qquad +K_p^2(h)h^{5/2}+R(h)h^3,
          	\end{align*}
          	where $K_p^2$ is some bounded function. This implies
          	\begin{equation}\label{eq:diffcovar}
          		S_{p,ij}^o-S_{p,ij}^s=\frac{3h^2}{8}\sum_{i=1}^d(1-|z_p^i|)\lambda_i D_i^2\L-\frac{1}{2} (\eta_p)_i(\eta_p)_j\partial_{ij}\L h^2+K_p^3h^{5/2}+R(h)h^3.
          	\end{equation}
          	Hence, we can replace $m_p(h)$ with $y_p$ and throw away
          	the terms involving $\partial_{ijkl}\phi$ in \eqref{eq:errortwoalg} with introducing errors at most $R(h)h^3$:
          	\[
          	|E|\le 
          	\Big|\frac{1}{2}\sum_{p\in P} w_p \partial_i\partial_j\phi(y_p)(S_{p,ij}^o-S_{p,ij}^s)\Big|+R(h)h^3.
          	\]
          	By \eqref{eq:diffcovar} and \eqref{eq:defeta}, we find
          	\begin{align*}
          		& \sum_{p\in P}  w_p \partial_i\partial_j\phi(y_p)(S_{p,ij}^o-S_{p,ij}^s) \nonumber \\
          		&  =\frac{h^2}{2}\partial_i\partial_j\phi(x_0)
          		\sum_{p\in P} w_p \big(\frac{3}{8}\sum_{i=1}^d(1-|z_p^i|)\lambda_i D_i^2\L-\frac{3}{4}\sum_{m,n}z_p^mz_p^n \sqrt{\lambda_m\lambda_n} D_mD_n\L \big)  +R(h)h^{5/2}.
          	\end{align*}
We note first that
\begin{equation}\label{eq:auxiliaryidentityapp}
\sum_{p\in P} w_p (1-|z_p^i|)=w^0=\frac{2}{3},\,  ~~\sum_{p\in P} w_p z_p^m z_p^n= 2w^1 \delta_{mn}=\frac{1}{3}\delta_{mn}.
\end{equation}
We justify the second equality for an example. 
Let $j\in\{\pm 1, 0\}$ be the index over the beams in one dimension,
and $z_{j}=j$. Then, when $m=n$,
\[
\sum_{p\in P} w_p z_p^m z_p^n=(\sum_{j=-1}^1 |z_j|^2 w^j)\prod_{i\neq m}(\sum_{j=-1}^1 w^j)=\sum_{j=-1}^1 |z_j|^2 w^j=2w^1.
\]
When $m\neq n$, 
\[
\sum_{p\in P} w_p z_p^m z_p^n=(\sum_{j=-1}^1 z_j w^j)^2\prod_{i\neq m, i\neq n}(\sum_{j=-1}^1 w^j)=(\sum_{j=-1}^1 z_j w^j)^2=0.
\]
Using \eqref{eq:auxiliaryidentityapp}, we find that
          	\[
          	\sum_{p\in P} w_p \partial_i\partial_j\phi(y_p)(S_{p,ij}^o-S_{p,ij}^s)
          	=R(h)h^{5/2}.
          	\]
          	Therefore
          	\[
          	|E|\le R(h)h^{5/2}.
          	\]
          	However, we know that $E_1$ does not contain $h^{5/2}$ terms
          	while $E_2$ neither does because of the symmetry. Hence, $|E|\le R(h)h^3$, which finishes the proof.
          \end{proof}

         \bibliographystyle{siam}
         \bibliography{GaussianMixture.bib}

          \end{document}